\documentclass{amsart}
\usepackage{amscd,amssymb,amsopn,amsmath,amsthm,graphics,amsfonts,enumerate,verbatim,calc}
\usepackage[dvips]{graphicx}
\usepackage[colorlinks=true,linkcolor=red,citecolor=blue]{hyperref}
\usepackage[all]{xy}
\usepackage{mathrsfs}

\calclayout

\newcommand{\rt}{\rightarrow}
\newcommand{\lrt}{\longrightarrow}

\newcommand{\st}{\stackrel}

\newcommand{\la}{\lambda}
\newcommand{\La}{\Lambda}

\newcommand{\Z}{\mathbb{Z}}

\newcommand{\CC}{\mathcal{C} }

\newcommand{\CE}{\mathcal{E}}
\newcommand{\CF}{\mathcal{F} }

\newcommand{\CH}{\mathcal{H}}
\newcommand{\CI}{\mathcal{I} }

\newcommand{\CM}{\mathcal{M} }

\newcommand{\CP}{\mathcal{P} }

\newcommand{\CS}{\mathcal{S} }
\newcommand{\CT}{\mathcal{T} }

\newcommand{\CV}{\mathcal{V}}

\newcommand{\CX}{\mathcal{X} }

\newcommand{\X}{\mathbf{X}}

\newcommand{\Mod}{{\rm{Mod\mbox{-}}}}

\newcommand{\mmod}{{\rm{{mod\mbox{-}}}}}

\newcommand{\im}{{\rm{Im}}}

\newcommand{\Coker}{{\rm{Coker}}}
\newcommand{\Ker}{{\rm{Ker}}}

\newcommand{\rad}{{\rm{rad}}}

\newcommand{\Hom}{{\rm{Hom}}}
\newcommand{\Ext}{{\rm{Ext}}}

\theoremstyle{plain}
\newtheorem{theorem}{Theorem}[section]
\newtheorem{corollary}[theorem]{Corollary}
\newtheorem{lemma}[theorem]{Lemma}

\newtheorem{proposition}[theorem]{Proposition}

\theoremstyle{definition}
\newtheorem{definition}[theorem]{Definition}

\newtheorem{construction}[theorem]{Construction}

\newtheorem{remark}[theorem]{Remark}

\theoremstyle{plain}

\theoremstyle{definition}

\numberwithin{equation}{section}

\begin{document}

\title[From morphism categories to functor categories  ]{From morphism categories to functor categories }

\author[ Rasool Hafezi and Hossein Eshraghi  ]{ Rasool Hafezi and Hossein Eshraghi}
\dedicatory{}

\address{School of Mathematics and Statistics,
Nanjing University of Information Science \& Technology, Nanjing, Jiangsu 210044, P.\,R. China}
\email{hafezi@nuist.edu.cn}
\address{Department of Pure Mathematics, Faculty of Mathematical Sciences, University of Kashan, PO Box 87317-51167, Kashan, Iran}
\email{eshraghi@kashanu.ac.ir }

\makeatletter
\@namedef{subjclassname@2020}{%
  \textup{2020} Mathematics Subject Classification}
\makeatother

\subjclass[2020]{18A25, 16G70, 16G10}

\keywords{Functor Category, Morphism Category, Auslander-Reiten Components.}

\begin{abstract}
For a nice-enough category $\CC$, we construct both the morphism category ${\rm H}(\CC)$ of $\CC$ and the category $\mmod\CC$ of all finitely presented contravariant additive functors over $\CC$ with values in Abelian groups. The main theme of this paper, is to translate some representation-theoretic attributes back and forth from one category to the other. This process is done by using an appropriate functor between these two categories, an approach which seems quite promising in particular when we show that many of almost split sequences are preserved by this functor. We apply our results to the case of wide subcategories of module categories to obtain certain auto-equivalences over them. Another part of the paper deals with Auslander algebras arising from algebras of finite representation type. In fact, we apply our results to study the Auslander-Reiten translates of simple modules over such algebras. In the last parts, we try to recognize particular components in the stable Auslander-Reiten quiver of Auslander algebras arising from self-injective algebras of finite representation type.
\end{abstract}

\maketitle

\section{Introduction}
As a popular belief, it is said that the introduction of the language of functor categories to the study of categories of modules over rings dates back to Auslander and his colleagues' works. These works trace back mainly to the papers \cite{A65, A2, A76, AR74, ARVI}. In particular, Auslander's Formula that suggests to recover the category $\mmod\Lambda$ of finitely generated modules over an Artin algebra $\Lambda$ as the quotient

$$\mmod\Lambda\simeq \frac{\mmod(\mmod\Lambda)}{\{ F : F(\Lambda)=0\}}$$
\noindent deserves attention; here and throughout, $\mmod(\mmod\Lambda)$ denotes the category of additive contravariant coherent functors on $\mmod\Lambda$ with values in $\mathcal{A}b$, the category of Abelian groups. While talking about the exchange between two categories consisting objects that are apparently of different types, one expects to encounter with functors transferring from one category to the other. Concerning the morphism categories and the functor categories, such a study has initiated probably in \cite{A2}. Roughly, the general theme of the current paper is to figure out how some representation-theoretic attributes transfer between functor and morphism categories.  However, to be more precise, we prefer to provide a layout of the paper section by section. Prior to this, we want to point out that the morphism category of $\Lambda$ has on its own right been systematically studied from various aspects: deriving its Auslander-Reiten theory in the language of AR-theory of $\Lambda$ \cite{RS, XZZ, Es, HE}, using it to study the so-called Auslander algebras \cite{AR76}, and looking at a particular subcategory of it, namely the monomorphism category, in order to establish links to Gorenstein homological algebra \cite{Z, LZ, ZX}.

\vspace{.07 cm}

First of all, to keep the results as general as possible, we try to deal with the morphism category ${\rm H}(\mathcal{C})$ of a nice-enough category $\mathcal{C}$ (definitions are recalled later on). Namely, if we assume that $\mathcal{C}$ is an idempotent-complete additive category that admits pseudokernels then, in Section $3$, we endow ${\rm H}(\mathcal{C})$ with an exact structure defined by degree-wise split exact sequences in $\mathcal{C}$, denoted ${\rm H}^{\rm cw}(\mathcal{C})$. Even though such constructions have been considered in some particular cases, e.g. in \cite{Ba} where the category of morphisms between projective modules over an Artin algebra come to play, and also in \cite{BSZ} where complexes of fixed size have been considered, we do it in a most general possible circumstance as declared above. The motivation behind such considerations comes from two origins. Firstly, we look for a reasonable structure on ${\rm H}(\mathcal{C})$ with respect to which one may define almost split sequences. Note, secondly, that if one imposes tougher conditions on $\mathcal{C}$, for instance taking $\mathcal{C}$ to be an extension-closed subcategory of $\mmod\Lambda$, then ${\rm H}(\mathcal{C})$ inherits an exact structure as an extension-closed subcategory of the morphism category of $\Lambda$. So now a natural question arises: What are intrinsic similarities between these two exact structures on ${\rm H}(\mathcal{C})$?

\vspace{.07 cm}

To get more involved with the aforementioned question, we need to take a glance at the contents of Section $4$. For, we recall form \cite{A2} that there exists a functor $\Theta: {\rm H}(\mathcal{C})\lrt\mmod\mathcal{C}$, where $\mmod\mathcal{C}$ is the category of contravariant additive coherent functors on $\mathcal{C}$. The objective in Section $4$ is to study $\Theta$ form the point of view of Auslander-Reiten theory. We show that $\Theta$ induces an equivalence ${\rm H}(\mathcal{C})/\big<(M\rt 0), (M\st{1}\rt M)\big>\,\simeq\mmod\mathcal{C}$ where $M$ runs through the objects of $\mathcal{C}$. Using this, we show that ${\rm H}^{\rm cw}(\mathcal{C})$ admits almost split sequences whenever $\mathcal{C}$ is assumed to be a dualizing variety. Furthermore, to conquer the question posed above, it is shown that if $\mathcal{C}$ is an extension-closed dualizing subvariety of $\mmod\Lambda$ then, in many cases, the almost split sequences in ${\rm H}^{\rm cw}(\mathcal{C})$ and ${\rm H}(\mathcal{C})$ coincide. Not going off-topic, one more thing will be proved: $\Theta$ respects almost split sequences.

\vspace{.07 cm}

In Section $5$, we turn to apply some of the results to the case of wide subcategories. To illuminate the role and importance of wide subcategories of $\mmod\Lambda$, we must point out that such subcategories arise naturally in the study of $\tau$-tiling theory of $\Lambda$ \cite{AIR} and in connection with determination of certain torsion classes in $\mmod\Lambda$ \cite{MS}. These also play significant role in the study of certain classes of universal localizations over $\Lambda$ \cite{MS, HMV1, HMV2}. Such classes of modules also appear in classification problems for the so-called $\tau$-tilting finite algebras.
Among other things, for a given functorially finite  subcategory $\CX$ of $\mmod\Lambda$ we construct, based on our previous results, an auto-equivalence $\sigma_\CX:\CX\rt \CX$. When $\CX$ is assumed to be a functorially finite wide subcategory, we show that $\sigma_\CX$ coincides with the identity functor on $\CX$; this is done by applying a result from \cite{En} indicating that every functorially finite wide subcategory of $\mmod\Lambda$ might be recognized as the module category of some Artin algebra. Using this, one obtains an  exact sequence $$ 0 \rt (-, \tau_{\CX}(X))\rt D(P, -)\rt D(Q, -)\rt D(X, -)\rt 0$$
in $\mmod \CX$ for every indecomposable module $X$ which is not projective in $\mathcal{X}$; here $\tau_{\mathcal{X}}$ denotes the Auslander-Reiten translation of $\mathcal{X}$ and $P\rt Q\rt X\rt 0$ is the minimal projective presentation of $X$ with respect to $\mathcal{X}$. In this regard, recall that for a non-projective $\Lambda$-module $M$ with minimal projective presentation $P\rt Q\rt M\rt 0$, there exists an exact sequence $0\rt\tau(M)\rt \nu(P)\rt \nu(Q)\rt\nu(M)\rt 0$ where $\tau$ and $\nu$ stand respectively for the Auslander-Reiten translation and the Nakayama functor over $\mmod\Lambda$. Hence the aforementioned exact sequence of functors resembles, and generalizes, the latter one. We believe that these observations are convincing-enough to say that the rich treasury behind functorially finite subcategories of $\mmod\Lambda$ might be discovered further by applying some instruments from functor categories.

\vspace{.08 cm}

In Section $6$, we switch to algebras $\Lambda $ of finite representation type. The main impetus for such a study comes from the fact that in this case, one may construct the Auslander algebra ${\rm A}$ of $\Lambda$ which is, by definition, the endomorphism algebra of a representation-generator $M$ of $\Lambda$. Then there is a nice interpretation of the category $\mmod{\rm A}$ in terms of the functor category; namely, there is a categorical equivalence $\mmod{\rm A}\simeq\mmod(\mmod\Lambda)$. In the meanwhile, it is known \cite{A76} that simple functors over $\mmod\Lambda$ correspond bijectively to indecomposable $\Lambda$-modules. Hence this categorical equivalence provides a framework in which one tries to understand in more details the simple modules over ${\rm A}$ and its projectively stable version $\underline{\rm A}$. The results presented in this section come up by analyzing certain almost split sequences mainly provided in \cite{HE} and also in \cite{HZ}. The main results discover a relation between the (inverse) Auslander-Reiten translation of simple  $\underline{\rm A}$-(resp. ${\rm A}$-) modules and the cosyzygies (resp. syzygies) of simple  $\underline{\rm A}$-(resp. ${\rm A}$-) modules.

\vspace{.08 cm}

The last section is devoted to study certain components in the (stable) Auslander-Reiten quiver $\Gamma_{\rm A}$ of the Auslander algebra ${\rm A}$ whenever $\Lambda$ is self-injective of finite representation type. Note that recognition of such components have already been the subject of some earlier researches \cite{IPTZ}. To this end, we firstly deal with $\tau_{\mathcal{H}}$-periodic objects by invoking some almost split sequences already obtained in \cite{HE}. In this direction, it turns out that the auto-equivalence $\mathscr{A}=\nu\tau^3$ of the stable category $\underline{\mmod}\Lambda$, as defined in \cite{HE}, plays a significant role. In fact, we show that the existence of certain $\mathscr{A}$-periodic $\Lambda$- modules compels $\Gamma_{\rm A}$ to contain a finite oriented cycle as a subquiver, and in particular, makes ${\rm A}$ into an algebra of finite representation type. Another result asserts that for $\Lambda$ self-injective of finite representation type, any component $\Xi$ of the stable Auslander-Reiten quiver of ${\rm A}$ that contains a certain simple module is either infinite or is of the form $\Z\Delta/G$ for a Dynkin quiver $\Delta$ and an automorphism group $G$ of $\Z\Delta$; this is based on a structural theorem due to Liu \cite{L}.

\section{preliminaries and notation}
In this section, we collect very briefly some necessary background material of the paper. When required, explicit references are provided.

\subsection{Functor Categories}\label{subsection-functor-category}


Let $k$ be a commutative Artinian ring and let $\CC$ be a  $k$-linear Krull-Schmidt category.
A \emph{$\CC$-module} is a contravariant additive functor from $\CC$ to the category
$\mathcal{A}b$ of Abelian groups. We denote by $\Mod\CC$ the category of all $\CC$-modules,
and by $\mmod\CC$ the full subcategory of $\Mod\CC$ consisting of finitely presented
modules. Recall from \cite{A65} that a $\CC$-module $M$ is called \emph{finitely presented} if there exists an
exact sequence
\[\Hom_\CC(-,A)\to\Hom_{\CC}(-,B)\to M\to0\]
in $\Mod\CC$, for some objects $A,B$ of $\CC$.
Moreover, $\rm{proj}\mbox{-}\CC$ and $\rm{inj}\mbox{-}\CC$ denote the full subcategories of $\mmod\CC$ consisting   of projective and injective objects in $\mmod\CC$, respectively. The category $\mmod \CC$ is an abelian category if and only if $\CC$ admits pseudokernels; see page $315$ of \cite{AR74}. We  shall sometimes write $ (-, X)$ to indicate the representable functor $\CC(-, X)=\Hom_{\CC}(-, X)$.

\subsection{Dualizing $k$-varieties}
Let $\frak{r}$ be the radical of $k$ and ${\rm E}(k/{\frak{r}})$ be the injective envelope of the $k$-module $k/{\frak{r}}$. A Hom-finite $k$-linear Krull-Schmidt category $\CC$ is called a \emph{dualizing $k$-variety} \cite{AR74}
if the $k$-dual functors $D:\Mod\CC\to \Mod(\CC^{\rm op})$ and $D:\Mod(\CC^{\rm op})\to \Mod\CC$ given by $D(F)(C)=\Hom_k(F(C), {\rm E}(k/{\frak{r}}))$ for every object $C$ of $\CC$ and $F\in\Mod(\CC)$ or $\Mod(\CC^{\rm op})$ induce dualities
\[D:\mmod\CC\to\mmod(\CC^{\rm op})\ \mbox{ and }\ D:\mmod(\CC^{\rm op})\to \mmod\CC.\]
In this case, it turns out that $\mmod\CC$ is an abelian subcategory of $\Mod\CC$ that admits enough projective and enough injective
objects \cite[Theorem 2.4]{AR74}. As an example, $\rm{proj}\mbox{-}\La$, the category of finitely generated projective modules over an Artin $k$-algebra $\La$, is a dualizing $k$-variety. We note from \cite[Proposition 2.6]{AR74} that if $\CC$ is a dualizing $k$-variety then so is $\mmod\CC$.
Furthermore, any functorially finite subcategory of a dualizing $k$-variety is itself a
dualizing $k$-variety by \cite[Theorem 2.3]{AS81}.

\subsection{Morphism Categories}
Let $\CC$ be a category.
The morphism category $\rm{H}(\CC)$ of $\CC$ is a category whose objects are morphisms $f:X\rt Y$ in $\CC$, and whose morphisms are given by commutative diagrams. If we regard the morphism $f:X\rt Y$ as an object in $\rm{H}(\CC)$, we will usually present it as
 $(X\st{f}\rt Y)$. However, due to typographical considerations, we have to use also the vertical notation
 $\left(\begin{smallmatrix}
 X\\
 Y
 \end{smallmatrix}
 \right)_f$.
A morphism between the objects $(X\st{f}\rt Y)$ and $(X'\st{f'}\rt Y')$ is presented as
$(\sigma_1, \sigma_2): (X\st{f}\rt Y)\rt (X'\st{f'}\rt Y')$
or,   $\left(\begin{smallmatrix}
 \sigma_1\\
 \sigma_2
 \end{smallmatrix}
 \right): \left(\begin{smallmatrix}
 X\\
 Y
 \end{smallmatrix}
 \right)_f \rt \left(\begin{smallmatrix}
 X'\\
 Y'
 \end{smallmatrix}
 \right)_{f'}$, where $\sigma_1:X\rt X'$ and $\sigma_2:Y\rt Y'$ are morphisms in $\CC$ with $\sigma_2f=f'\sigma_1.$

\vspace{.1 cm}

Adapting the notation, the morphism category raised from $\mathcal{C}=\mmod\Lambda$, the category of finitely generated right modules over an Artin $k$-algebra $\Lambda$, will be denoted simply by $\mathcal{H}$; this will cause no ambiguity. The same rule also applies to the monomorphism category $\CS$ of $\Lambda$ whose objects are just monic $\Lambda$-maps.

\subsection{Auslander-Reiten-Serre Duality}
Let $(\mathcal{C}, \CE)$ be an exact category in the sense of Quillen \cite{Q, K} (see next section for an introduction). Recall that a morphism $v\colon E\to Y$ in $\mathcal{C}$ is called \emph{right almost split} if it is not a retraction and each $f\colon Z\to Y$ which is not a retraction factors through $v$. Dually, a morphism $u\colon X\to E$ in $\mathcal{C}$
is called \emph{left almost split} if it is not a section and each $f\colon X\to Z$ which is not a section factors through $u$. An admissible sequence  $\delta\colon 0 \rt X \xrightarrow{u} E \xrightarrow{v} Y \rt 0$ in $\CE$ is an \emph{almost split sequence} if $u$ is left almost split and $v$ is right almost split.  Since $\delta $ determines $X$ and $Z$ in a unique way, we call $X$ the Auslander-Reiten translation $X=\tau_{\mathcal{C}}(Y)$ of $Y$ in $\mathcal{C}$.  A non-zero object $X \in \mathcal{C}$ is said to be {\it endo-local} if its endomorphism ring $\rm{End}_{\mathcal{C}}(X)$ is local. Following \cite[Definition 3.1]{INY}, we say that $\mathcal{C}$ has almost split sequences if endo-local non projective objects of $\CC$ and endo-local non-injective objects of $\CC$ are respectivel the terminal and the initial terms of some almost split sequence in $\CE$.

\vspace{.1 cm}

Assume  now that $\mathcal{C}$ is further a {\it $k$-linear category} and let $D$ be the $k$-dual functor. Put $\underline{\CC}$ and $\overline{\CC}$ denote respectively the projectively and the injectively stable categories of $\CC$.  An Auslander-Reiten-Serre duality (ARS duality, in brief) is a pair $(\tau_{\CC}, \eta)$ consisting of an equivalence functor $\tau_{\mathcal{C}}: \underline{\mathcal{C}}\rt \overline{\mathcal{C}}$ together with a bi-natural isomorphism
$$\eta_{X, Y}:\Hom_{\underline{\mathcal{C}}}(X, Y)\simeq D\Ext^1_{\mathcal{C}}(Y, \tau_{\mathcal{C}}(X)) \ \ \ \text{for any}  \ \ X, Y \in \mathcal{C}.$$
The following lemma, taken from \cite[Theorem 3.6]{INY}  (see also \cite{J}), provides a close connection between the existence of almost split sequences in $\mathcal{C}$ and the existence of an ARS-duality. Let us recall that under the above hypothesis, $\mathcal{C}$ is {\it Ext-finite} if the $k$-modules $\Ext^1_{\mathcal{C}}(X, Y)$ are finitely generated.

\begin{lemma}\label{ARS duality}
	Let $\mathcal{C}$ be a $k$-linear Ext-finite Krull-Schmidt exact category. Then the following conditions are equivalent.
	\begin{itemize}
		\item [$(1)$] $\mathcal{C}$ has almost split sequences.
		\item [$(2)$] $\mathcal{C}$ has an Auslander-Reiten-Serre duality.
		\item [$(3)$] The stable category $\underline{\mathcal{C}}$ is a dualizing $k$-variety.
		\item [$(4)$] The stable category $\overline{\mathcal{C}}$ is a dualizing $k$-variety.
	\end{itemize}	
\end{lemma}

Throughout the paper, $\Lambda$ will stand for a fixed Artin $k$-algebra and modules are, by default, finitely generated right modules. The Auslander-Reiten translation, the Nakayama functor, the syzygy and the cosyzygy functor of $\Lambda$ are respectively denoted by $\tau$, $\nu$, $\Omega$, and $\Omega^{-1}$. If we deal with an algebra other than $\Lambda$ or with a category, these functors will be accompanied with necessary subscripts. The symbols $\Ker$, $\Coker$, and $\im$, used freely in all contexts, stand respectively for the kernel, cokernel, and the image of morphisms.

\section{Exact structures on the morphism category}
An exact category  $(\CC, \CE)$ is formed by an additive category $\CC$, and a class $\CE$ of composable pairs of morphisms in $\CC$ (also called kernel-cokernel pairs) satisfying certain axioms that we refrain to exhibit here and refer the reader e.g. to \cite{K}. The composable pair $(i, p)$ in $\CE$ is usually denoted by $0 \rt A' \st{i}\rt  A \st{p}\rt  A'' \rt 0$, where $i: A' \rt A$ and $p: A \rt A''$ are respectively called an $\CE$-admissible monic and an $\CE$-admissible epic. Composable pairs, admissible monics and admissible epics are sometimes referred to respectively as conflations, inflations and deflations. The notion of an exact category was first introduced by Quillen in \cite{Q} and then Keller  \cite{K} proved the redundancy of some axioms.

\vspace{.1 cm}
Let $\CC$ be an additive category. In this section, we shall put exact structures on the morphism category ${\rm H}(\CC)$ of $\CC$ \cite{Ba}. In this connection, let us mention that such studies have also been conducted in \cite{BSZ} and the material presented in this section contains a detailed study of that work if one restricts to the case of complexes of size $2$ in the sense defined there.

\vspace{.2 cm}
{\it Exact structure {\bf E1}.} Let  $\CE^{\rm cw}$ be the class of all pairs of composable morphisms
$$\xymatrix@1{ \delta: & {\left(\begin{smallmatrix} X_1\\ X_2\end{smallmatrix}\right)}_{f}
		\ar[rr]^-{\left(\begin{smallmatrix} \phi_1 \\ \phi_2 \end{smallmatrix}\right)}
		& & {\left(\begin{smallmatrix}Z_1\\ Z_2\end{smallmatrix}\right)}_{h}\ar[rr]^-{\left(\begin{smallmatrix} \psi_1 \\ \psi_2\end{smallmatrix}\right)}& &{\left(\begin{smallmatrix}Y_1 \\ Y_2\end{smallmatrix}\right)}_{g}&  } \ \    $$
such that the induced composable morphisms $X_i\st{\phi_i}\rt Z_i\st{\psi_i}\rt Y_i$ split in $\CC$ for $i=1, 2$.
It can be easily seen that any pair of composable morphisms in $\CE^{\rm cw}$ is isomorphic to a pair of composable morphisms of the form
	$$\xymatrix@1{ \delta': & {\left(\begin{smallmatrix} X_1\\ X_2\end{smallmatrix}\right)}_{f}
		\ar[rr]^-{\left(\begin{smallmatrix} {\left[\begin{smallmatrix} 1\\ 0\end{smallmatrix}\right]} \\ {\left[\begin{smallmatrix} 1\\ 0\end{smallmatrix}\right]} \end{smallmatrix}\right)}
		& & {\left(\begin{smallmatrix}X_1\oplus Y_1\\ X_2\oplus Y_2\end{smallmatrix}\right)}_{h}\ar[rr]^-{\left(\begin{smallmatrix} [0~~1] \\ [0~~1]\end{smallmatrix}\right)}& &
		{\left(\begin{smallmatrix}Y_1  \\ Y_2\end{smallmatrix}\right)}_{g}&  } \ \    $$
	where $h={\left(\begin{smallmatrix}f & q\\ 0 & g \end{smallmatrix}\right)}$ and $q:Y_1\rt X_2$ is a possibly non-zero morphism in $\CC$. Regarding this easy observation, without loss of generality, we usually take all kernel-cokernel pairs in ${\rm H}(\CC)$ to be of this form; this is justified by the following lemma.
	
\begin{lemma}\label{Lkercoker}
Any object in $\CE^{\rm cw}$ is a  kernel-cokernel pair in ${\rm H}(\CC)$.
 \end{lemma}	
\begin{proof}
 Take the element $\delta'$ of $\CE^{\rm cw}$ and assume that the composite of the morphisms $(\sigma_1, \sigma_2):(X_1\oplus Y_1\st{h}\rt X_2\oplus Y_2)\rt (V\st{s}\rt W)$ and $({\left[\begin{smallmatrix} 1\\ 0\end{smallmatrix}\right]}, {\left[\begin{smallmatrix} 1\\ 0\end{smallmatrix}\right]})$ vanishes. This means that the restriction of $\sigma_i$ on $X_i$, for $i=1, 2$, is the zero map. This enables us to define the morphisms $\sigma_1|_{Y_1}$ and  $\sigma_2|_{Y_2}$ and it readily follows that $(\sigma_1, \sigma_2)$ factors uniquely over $([0~~1], [0~~1])$ via the morphism $(\sigma_1|_{Y_1}, \sigma_2|_{Y_2})$.
The remaining axioms are verified similarly.
\end{proof}	
	
Recall that an additive category $\mathcal{D}$ is called
idempotent-complete if every idempotent endomorphism in $\mathcal{D}$ admits a kernel.

\begin{proposition}\label{exact structure on H}
Assume $\mathcal{C}$ is idempotent-complete and admits pseudokernels. Then $\CE^{\rm cw}$ defines an exact structure on the additive category ${\rm H}(\CC)$.
\end{proposition}
\begin{proof}
Since $\CC$ is idempotent-complete, it is known that the Yoneda functor gives an equivalence $\CC\simeq\rm{proj}\mbox{-}\CC$. This equivalence is  naturally extended to an equivalence between corresponding morphism categories; i.e., ${\rm H}(\CC)\simeq {\rm H}({\rm proj}\mbox{-}\CC)$. One observes that, under this equivalence, the kernel-cokernel pairs in $\CE^{\rm cw}$ provided by Lemma \ref{Lkercoker} correspond bijectively to the short exact sequences in the abelian category ${\rm H}(\mmod \CC)$ whose terms lie inside ${\rm H}({\rm proj}\mbox{-}\CC)$. But the subcategory ${\rm H}({\rm proj}\mbox{-}\CC)$ is closed under extensions and inherits an exact structure from ${\rm H}(\mmod \CC)$.
\end{proof}

\begin{remark}
To make our arguments work, we had to impose some restrictions on the additive category $\CC$ to get a suitable exact structure out of ${\rm H}(\CC)$. However,  it may be the case that the aforementioned set of requirements is not minimal in the sense that the above family of kernel-cokernel pairs may equip ${\rm H}(\CC)$ with an exact structure even if some of the hypothesis in Proposition \ref{exact structure on H} are dropped.
\end{remark}

From now on we assume  that $\CC$ is idempotent-complete and  admits pseudokernels, and the symbol
${\rm H}^{\rm cw}(\CC)$ stands for the exact category ${\bf E1}=({\rm H}(\CC), \CE^{\rm cw})$, sometimes also called the cw-exact category. The following proposition is recorded for future use.

\begin{proposition}\label{projective-object}
	Suppose $(X_1 \st{f} \rt X_2)$ is an object in ${\rm H}^{\rm cw}(\CC)$.
	\begin{itemize}
		\item [$(1)$] $f$ defines an indecomposable projective object in ${\rm H}^{\rm cw}(\CC)$ if and only if it is isomorphic either to $(X \st{1}\rt X)$ or $(0 \rt X)$ for some  indecomposable object $X$ in $\CC$.
	\item[$(2)$]  $f$ defines an indecomposable injective object in ${\rm H}^{\rm cw}(\CC)$ if and only if it is isomorphic either to $(X \st{1}\rt X)$ or $(X\rt 0)$ for some  indecomposable object $X$ in $\CC$.
	\end{itemize}
Furthermore, ${\rm H}^{\rm cw}(\CC)$ has enough projectives and enough injectives.
\end{proposition}
\begin{proof}
This should be compared to \cite[Corollary 3.2]{Ba}. We just remark that the last claim follows from the short exact sequences
  $$\xymatrix@1{0\ar[r] & {\left(\begin{smallmatrix} 0\\ X_1\end{smallmatrix}\right)}_{0}
		\ar[rr]^-{\left(\begin{smallmatrix} 0 \\ {\left[\begin{smallmatrix}f\\
			-1\end{smallmatrix}\right]}
			\end{smallmatrix}\right)}
		& & {\left(\begin{smallmatrix}0\\ X_2\end{smallmatrix}\right)}_{0}\oplus{\left(\begin{smallmatrix}X_1\\
			X_1\end{smallmatrix}\right)}_{1}\ar[rr]^-{\left(\begin{smallmatrix}1\\\left[\begin{smallmatrix} 1 &f
			\end{smallmatrix}\right]
			\end{smallmatrix}\right)}& &
		{\left(\begin{smallmatrix}X_1\\X_2\end{smallmatrix}\right)}_{f}\ar[r]& 0 }$$
and
$$\xymatrix@1{0\ar[r] & {\left(\begin{smallmatrix}X_1\\X_2\end{smallmatrix}\right)_f}\ar[rr]^{\left(\begin{smallmatrix}{\left[ \begin{smallmatrix}f\\ 1\end{smallmatrix}\right]} \\ 1\end{smallmatrix}\right)}
		& & {\left(\begin{smallmatrix}X_2\\
			X_2\end{smallmatrix}\right)_{1}\oplus{\left(\begin{smallmatrix} X_1\\ 0\end{smallmatrix}\right)_{0}}}\ar[rr]^-{\left(\begin{smallmatrix}{\left[\begin{smallmatrix}-1 & f\end{smallmatrix}\right]}\\ 0
			\end{smallmatrix}\right)} & &
		{\left(\begin{smallmatrix}X_2\\ 0\end{smallmatrix}\right)}\ar[r] & 0} $$
in $\mathcal{E}^{\rm{cw}}$.
\end{proof}

\vspace{.2 cm}
{\it Exact structure {\bf E2}}. Now assume $\CC$ is an extension-closed subcategory of $\mmod\Lambda$ for an Artin algebra $\Lambda$. We may consider $\CC$ as an exact category through the structure induced by the abelian category $\mmod \La$. Then also ${\rm H}(\CC)$, as an extension-closed subcategory of the abelian category $\CH$ is endowed  with the {\it canonical exact structure} inherited from $\CH$, still denoted by ${\rm H}(\CC)$. We also keep the cw-exact structure ${\rm H}^{\rm cw}(\CC)$ defined by degree-wise split sequences. It will be indicated in the next section that if $\CC$ is a $k$-dualizing variety, then ${\rm H}^{\rm cw}(\CC)$ admits almost split sequences. Further, the canonical exact category ${\bf E2}={\rm H}(\CC)$ admits almost split sequences provided $\CC$ is a $k$-dualizing subvariety of $\mmod\Lambda$. It also becomes clear how the canonical exact category ${\rm H}(\CC)$ inherits almost split sequence from ${\rm H}^{\rm cw}(\CC)$ in the latter case. However, for technical reasons, we have to defer the proofs until next section.

\begin{remark} Suppose for a moment that $\CC$ is further functorially finite in $\mmod\Lambda$. In this case, another approach one may take to show that ${\rm H}(\CC)$ has almost split sequences is to explore when ${\rm H}(\CC)$ is functorially finite in $\CH$. This seems natural in view of the fact that, by \cite[Theorem 2.4]{AS81}, any functorially finite extension-closed subcategory of $\CH$ admits almost split sequences.
\end{remark}

\vspace{.1 cm}

{\it Exact structure {\bf E3}}. Restricting to the case where $\CC=\mmod\Lambda$, in the  last part of this section, we put a third exact structure on $\CH$ that will turn out in Section $7$ to be in connection with the stable Auslander-Reiten quiver of Auslander algebras; see Remark \ref{RMark-Frob}. 

\begin{definition} An indecomposable object in $\CH$ is said to be of type $(a)$ (resp. $(b)$, or $(c)$) provided it is isomorphic to $(0 \rt M)$ (resp. $(M\st{1}\rt M)$, or $(M\rt 0)$) for some $\Lambda$-module $M$. Further, an indecomposable object is said to be of type $(d)$ if it is isomorphic to $(P\st{f}\rt Q)$  where $P, Q$ are projective $\Lambda$-modules. 
\end{definition}

Let now $\CX$ be the smallest subcategory of $\CH$ containing  all objects of types $(a), (b), (c)$ and $(d)$. Let also $\CE_{\CX}$ be the class of all short exact sequences $0 \rt {\rm X}\rt {\rm Y}\rt {\rm Z}\rt 0$ in $\CH$  such that the induced sequence
$$0 \rt \Hom_{\CH}({\rm V} , {\rm X})\rt \Hom_{\CH}({\rm V}, {\rm Y})\rt \Hom_{\CH}({\rm V}, {\rm Z})\rt 0$$
is exact for every $V \in \CX$. We know from \cite{AS93} and \cite{Bu} that $\CE_{\CX}$ induces an exact structure on $\CH$ which is denoted by ${\bf E3}=\CH_{\CX}= (\CH, \CE_{\CX})$. One infers from \cite[Theorem 1.12]{AS93} that the exact category $\CH_{\CX}$ has enough projectives and enough injectives.  Denote by $\CP(\CH_{\CX})$ (resp. $\CI(\CH_{\CX})$) the subcategory of projective (resp. injective) objects in $\CH_{\CX}$. In view of \cite[Corollary 1.6 and  Proposition 1.10]{AS93},  we have
$\CP(\CH_{\CX})=\CX\cup {\rm proj}\mbox{-}\CH$ and  $\CI(\CH_{\CX})=\tau_{\CH}(\CX)\cup {\rm inj}\mbox{-}\CH$, where ${\rm proj}\mbox{-}\CH$ and ${\rm inj}\mbox{-}\CH$ stand respectively for the subcategories of projective and injective objects in $\mathcal{H}$ and $\tau_\CH$ is the Auslander-Reiten translation of $\CH$.  We exploit \cite[Proposition 1.9]{AS93} to examine the almost split sequences in $\CH_{\CX}$;  it turns out that an almost split sequence $0 \rt {\rm X}\rt {\rm Y}\rt {\rm Z}\rt 0$ in $\CH$ is an almost split sequence in $\CH_{\CX}$ if and only if neither ${\rm X} \in \CI({\CH}_{\CX})$ nor ${\rm Z} \in \CP({\CH}_{\CX})$.

\section{Interplay between morphism and functor categories}

Until further notice, we assume throughout the section that $\CC$ is a dualizing $k$-variety. In this section, we will be involved with a functor going from morphism category to the functor category, originally defined and studied in \cite{A2} and then reconsidered in \cite{HM}.  This functor is our main tool to exchange between these two categories. The construction is based on the Yoneda functor.

\begin{construction}\label{Construction-Morphism-Functor}
 Let $(X_1\st{f}\rt X_2)$ be an object of ${\rm H}(\CC)$. Define
$$ (X_1\st{f}\rt X_2)\st{\Theta}\mapsto\ \Coker(\CC(-, X_1)\st{\CC(-, f)}\lrt \CC(-, X_2)).$$	
If $h=\left(\begin{smallmatrix}
	h_1\\ h_2
	\end{smallmatrix}\right):X=\left(\begin{smallmatrix}
	X_1\\ X_2
	\end{smallmatrix}\right)_f\rt \left(\begin{smallmatrix}
	X'_1\\ X'_2
	\end{smallmatrix} \right)_{f'}=X'$  is a morphism in $\rm{H}(\La)$, then we let $\Theta (h)$ be the unique morphism $\sigma$ that makes the following diagram commute.
	
	\[\xymatrix{\Hom_{\CC}(-, X_1) \ar[rr]^{\Hom_{\CC}(-, f)} \ar[d]^{\Hom_{\CC}(-, h_1)} && \Hom_{\CC}(-, X_2) \ar[r] \ar[d]^{\Hom_{\CC}(-, h_2)} & \Theta(\mathrm{X})  \ar[r] \ar[d]^{\sigma} & 0 \\
	\Hom_{\CC}(-, X'_1)  \ar[rr]^{\Hom_{\CC}(-,f')} && \Hom_{\CC}(-, X'_2) \ar[r] & \Theta(\mathrm{X}')  \ar[r]  & 0.}\]
\end{construction}

\noindent It is routine to verify that these rules introduce a well-defined functor $\Theta: {\rm H}(\CC)\rt \mmod \CC$ (In forthcoming sections, we will use the symbol $\Theta_\CC $ to emphasize on $\CC$.) The purpose of this section is to study this functor from the perspective of almost split sequences. It turns out that $\Theta$ behaves well over such sequences. Firstly, we need to recall some facts on objective functors; more details are provided in the Appendix of \cite{RZ}. Let $F : \CC \lrt \mathcal{D}$ be an additive functor between additive categories. $F$ is called an objective functor if any morphism $f$ in $\CC$ with $F(f) = 0$ factors through an object $K$ of $\CC$ with $F(K) = 0$; such a $K$ is then called a kernel object of $F$. We say that the kernel of a functor $F$ is generated by a subcategory $\CX$ of $\CC$ if ${\rm add}\mbox{-}\CX$, the additive closure of $\CX$ in $\CC$, is the class of all kernel objects of $F$.

\vspace{.05 cm}

Let $F : \CC \lrt \mathcal{D}$ be a full, dense and objective functor and let the kernel of $F$ be generated by $\CX$. Then $F$ induces an equivalence $\overline{F}: \CC/ \CX  \lrt \mathcal{D}$ where the additive quotient category $\CC/ \mathcal{X}$ of $\CC$ with respect to $\mathcal{X}$ has the same objects as $\CC$ and the morphisms are defined via the rule
$$\CC/\mathcal{X}(X, Y):=\CC(X, Y)/\{\phi\mid \phi \ \text{factors through an object in } \ {\rm add}\mbox{-}\CX \}$$
for any pair of objects $X,Y$ of $\CC$.

\begin{theorem}\label{Theorem-morphism-functor}
The functor $\Theta:{\rm H}(\CC) \lrt \mmod \CC$ is full, dense and objective.  Thus, there exists an equivalence $\overline{\Theta}$ of categories that makes the following diagram commute.
\[\xymatrix{
{\rm H}(\CC)\ar[r]^{\Theta}\ar[d]^{\pi} & \mmod\CC\\
\frac{{\rm H}(\CC)}{\CV}\ar[ur]^{\overline{\Theta}}   \\
}\]
\noindent
Here, $\pi$ is the natural quotient map and $\CV$ is the full subcategory of ${\rm H}(\CC)$ generated by all finite direct sums of objects of type $(b)$ or $(c)$, that is to say, objects of the form $(M\st{1}\lrt M)$ and $(M\lrt 0)$, where  $M$ runs through the objects of $\CC$.
\end{theorem}

\begin{proof}
$\Theta$ is dense; for take $F \in \mmod \CC$ with a projective presentation $(-, X) \st{(-, g)} \rt (-, Y) \rt F \rt 0$.  It is plain that $\Theta(X\st{g}\rt Y)=F.$
To see the fullness of $\Theta$, take two objects $(X\st{g}\rt Y)$ and $(X'\st{g'}\rt Y')$ of  ${\rm H}(\CC)$. As the representable functors $(-, Y)$ and $(-, Y')$ are projective, it follows that any morphism $\sigma: F=\Theta(X\st{g}\rt Y)\rt\Theta(X'\st{g'}\rt Y')=F'$ in $\mmod \CC$ might be lifted to a map from the augmented projective presentation $(-, X)\st{(-,g)}\rt(-, Y)\rt F\rt 0$ to $(-, X')\st{(-,g')}\rt(-, Y')\rt F'\rt 0$. Then using Yoneda's Lemma and the aforementioned construction, one obtains a morphism $h:(X\st{g}\rt Y)\rt (X'\st{g'}\rt Y')$ in ${\rm H}(\CC)$ with $\sigma=\Theta(h)$.

\vspace{.1 cm}

Now assume $\Theta(X \st{g} \rt Y)=0$, for some object $(X\st{g}\rt Y)$. Then by the construction, we have the exact sequence
$0 \rt (-, X) \st{(-, g)} \lrt (-, Y)  \rt 0$  in $\mmod \mathcal{C}$. One then observes that the identity map $1:Y\rt Y$ factors over $g$ via, say, $h:Y\rt X$. Therefore, $X={\rm Im} (h) \oplus {\rm Ker} (g)$. This leads to the decomposition $(X\st{g}\rt Y)= ({\rm Ker}(g)\rt 0) \oplus ({\rm Im}(h) \st{g \mid} \rt Y)$ where $g|$ is the restricted map which must be an isomorphism since $gh=1_Y$. This settles that the kernel of $\Theta$ is generated by $\CV$.
	
\vspace{.1 cm}

Finally, suppose $\Theta(h)=0$, for $h=(h, h'): (X\st{g} \rt Y) \rt (X' \st{g'} \rt Y')$ in ${\rm H}(\CC)$. Setting $F=\Theta(X\st{g}\rt Y)$ and $F'=\Theta(X'\st{g'}\rt Y')$, this induces a chain map between complexes of functors
	
\[\xymatrix{  \cdots\ar[r] & (-, Z_0)\ar[d]_{(-,\alpha_0)}\ar[r] & (-, X) \ar[d]^{(-, h)} \ar[r]^{(-, g)} & (-, Y) \ar[d]^{(-, h')} \ar[r] & F \ar[d]^{0} \ar[r] & 0 \\
\cdots\ar[r] & (-, Z'_0)\ar[r] & (-, X')\ar[r]^{(-, g')} & (-,Y' )\ar[r] & F' \ar[r] & 0}\]
raised by taking projective presentations of $F$ and $F'$. Evidently, this chain map is null-homotopic and, according to \cite[Corollary 3.5]{G}, factors through a contractible complex. As any contractible complex of functors might be imagined to be a direct sum of complexes of the form
$$\cdots\rt 0\rt (-, W)\st{1}\rt (-, W)\rt 0\rt\cdots$$ for various objects $W$ of $\CC$, this induces a commutative diagram

$$\xymatrix{    & \cdots \ar[r] & (-, Z_0)   \ar@/_1.25pc/@{.>}[dd]_<<<<<<{(-, \alpha_0)} \ar[d]  \ar[r]  & (-, X)  \ar@/_1pc/@{.>}[dd]_<<<<<<<<{(-, h)} \ar[d]  \ar[r]^{(-, g)}  &  (-, Y)  \ar@/_1.2pc/@{.>}[dd]_<<<<<{(-, h')} \ar[d]  \ar[r]  & 0    \\
		&\cdots \ar[r] & (-, Z_0\oplus X) \ar[r] \ar[d] & (-, X\oplus Y) \ar[r] \ar[d]  &   (-, Y)  \ar[d]  \ar[r]  & 0 \\
		& \cdots \ar[r] & (-, Z'_0)  \ar[r]  & (-, X')  \ar[r]^{(-, g')}  &(-,Y') \ar[r] & 0. }$$

\noindent Therefore, there exists a factorization of $h$ through the object $(X\rt 0)\oplus (Y\st{1}\rt Y)$, which is a kernel object according to the above paragraph. This shows that $\Theta$ is an objective functor. Now the existence of the equivalence $\overline{\Theta}$ comes up from observations prior to the theorem.
\end{proof}


Let us record here that applying a dual construction to the opposite category $\CC^{\rm op}$ results in a contravariant functor
$$\Theta':{\rm H}(\CC)\rt \mmod \CC^{\rm op}, \ \ \ (X\st{f}\rt Y)\mapsto  \Coker( \CC(Y, -)\st{\CC(f, -)}\lrt \CC(X, -)) $$
which is seen to induce a duality
$\overline{\Theta'}$ that makes the diagram
\[\xymatrix{
{\rm H}(\CC)\ar[r]^{\Theta'}\ar[d]^{\pi'} & \mmod\CC^{\rm op}\\
\frac{{\rm H}(\CC)}{\CV'}\ar[ur]^{\overline{\Theta'}}   \\
}\]
\noindent commute. Here, $\CV'$ is the full subcategory of ${\rm H}(\CC)$ generated by all finite direct sums of objects of type $(a)$ or $(b)$.

\vspace{.1 cm}

Consider the morphism category ${\rm H}(\CC)$, endowed with the exact structure given by $\CE^{\rm cw}$. According to Proposition \ref{projective-object}, $\CV$ (resp. $\CV'$) is nothing but the subcategory of injective (resp. projective) objects of the exact category ${\rm H}^{\rm cw}(\CC)$. Consequently, the factor categories $\rm{H}(\CC)/\CV'$ and $\rm{H}(\CC)/\CV$ are equivalent respectively to the projectively and injectively stable categories $\underline{{\rm H}^{\rm cw}}(\CC)$ and $\overline{{\rm H}^{\rm cw}}(\CC)$. Hence the following proposition emerges to settle that ${\rm H}^{\rm cw}(\CC)$ admits almost split sequences.

\begin{proposition}\label{Prop.exsiyences-AR-comwise}
For a dualizing $k$-variety $\CC$, the following statements hold.
\begin{itemize}
\item [$(1)$] ${\rm H}^{\rm cw}(\CC)$ admits almost split sequences.
\item [$(2)$] ${\rm H}^{\rm cw}(\CC)$ has an Auslander-Reiten-Serre duality.
\end{itemize}
\end{proposition}
\begin{proof}
The above observations along with Theorem \ref{Theorem-morphism-functor} provide the equivalences ${\rm H}(\CC)/\CV \simeq \mmod \CC\simeq \overline{{\rm H}^{\rm cw}}(\CC)$. Note that by \cite[Proposition 2.6]{AR74}, $\mmod \CC$ is a dualizing $k$-variety as well. Hence Lemma \ref{ARS duality} completes the proof.
\end{proof}

\begin{remark}
We already know from Proposition \ref{exact structure on H} that ${\rm H}^{\rm cw}(\CC)$ enjoys an exact structure ${\bf E1}$ defined by component-wise split sequences. We want to point out that, in general, $\Theta$ does not respect this structure. Said in other words, $\Theta$ is not an exact functor with respect to this exact structure. To see this, take an element $(X\st{f}\rt Y)$ where $X,Y$ are non-zero objects of $\CC$. Then 

$$\xymatrix@1{ {\left(\begin{smallmatrix} 0\\ Y\end{smallmatrix}\right)}_{0}
		\ar[rr]^-{\left(\begin{smallmatrix} 0 \\ 1 \end{smallmatrix}\right)}
		& & {\left(\begin{smallmatrix}X\\ Y\end{smallmatrix}\right)}_{f}\ar[rr]^-{\left(\begin{smallmatrix} 1 \\ 0
	\end{smallmatrix}\right)}& &{\left(\begin{smallmatrix}X \\ 0\end{smallmatrix}\right)}_{0}&  } \ \    $$
lies inside ${\CE}^{\rm cw}$. Applying $\Theta$ gives rise to the sequence $0\rt \CC(-, Y)\rt \Coker(\CC(-, f))\rt 0$ whose exactness would, according to the definition of $\Theta$, imply $X=0$.  
\end{remark}

We are now in a position to prove the assertion in previous section concerning the existence of almost split sequences in ${\rm H}(\CC)$ where $\CC$ is an extension-closed dualizing subvariety of $\mmod\Lambda$ for an Artin algebra $\Lambda$. The following lemma which is taken from \cite[Proposition 3.1]{MO} will be fruitful. While the special case $\CC=\mmod \La$ has been dealt with in \cite{MO}, the same proof still works for general $\CC$ as we consider here.

\begin{lemma}\label{LemmaBoundary}
	Assume  $\delta:  0 \rt  A\st{f} \rt B\st{g} \rt C \rt 0$ is  an almost split sequence in $\CC$. Then
	\begin{itemize}
		\item[$(1)$] The almost split sequence in ${\rm H}(\CC)$ ending at $(0\rt C)$ is of the form
			$$\xymatrix@1{  0\ar[r] & {\left(\begin{smallmatrix} A\\ A\end{smallmatrix}\right)}_{1}
			\ar[rr]^-{\left(\begin{smallmatrix} 1 \\ f\end{smallmatrix}\right)}
			& & {\left(\begin{smallmatrix}A\\ B\end{smallmatrix}\right)}_{f}\ar[rr]^-{\left(\begin{smallmatrix} 0 \\ g\end{smallmatrix}\right)}& &
			{\left(\begin{smallmatrix}0\\ C\end{smallmatrix}\right)}_{0}\ar[r]& 0. } \ \    $$		
	
		\item [$(2)$] The almost split sequence in ${\rm H}(\CC)$  ending at $(C\st{1}\rt C)$ is of the form
			$$\xymatrix@1{  0\ar[r] & {\left(\begin{smallmatrix} A\\ 0\end{smallmatrix}\right)}_{0}
			\ar[rr]^-{\left(\begin{smallmatrix} f \\ 0 \end{smallmatrix}\right)}
			& & {\left(\begin{smallmatrix} B\\ C\end{smallmatrix}\right)}_{g}\ar[rr]^-{\left(\begin{smallmatrix} g \\ 1\end{smallmatrix}\right)}& &
			{\left(\begin{smallmatrix}C \\ C\end{smallmatrix}\right)}_{1}\ar[r]& 0. } \ \    $$		
		\end{itemize}
\end{lemma}

\begin{proposition}\label{p almost split seq}
Let $\CC$ be an extension-closed $k$-dualizing subvariety of $\mmod\Lambda$. Then the canonical exact category  ${\bf E2}={\rm H}(\CC)$ admits almost split sequences.
\end{proposition}
\begin{proof}
Let $ \mathrm{Z}$ be an indecomposable non-projective object in ${\rm H}(\CC)$. Assume first that $\mathrm{Z}$ is of either types $(0\rt X)$  or $(X\st{1}\rt X)$, for an object $X\in\CC$. Then since $\CC$ admits almost split sequences by \cite[Theorem 1.1]{AS81}, from Lemma \ref{LemmaBoundary} we infer that $\mathrm{Z}$ is the end term of an almost split sequence in ${\rm H}(\CC)$. Otherwise, $\mathrm{Z}$ is not projective in the exact category ${\rm H}^{\rm cw}(\CC)$ by Proposition \ref{projective-object}. Hence, by Proposition \ref{Prop.exsiyences-AR-comwise}, there exists  an almost split sequence ending at $\mathrm{Z}$ in the exact category ${\rm H}^{\rm cw}(\CC)$. However, following the definitions, it is easy to verify that this is an almost split sequence in ${\rm H}(\CC)$ as well.
\end{proof}

The following corollary is an immediate consequence of the arguments above.

\begin{corollary}\label{cor dg slplit almost split}
 Let $\CC$ be an extension-closed $k$-dualizing subvariety of $\mmod\Lambda$ and let
 $$\xymatrix@1{ 0 \ar[r] & {\left(\begin{smallmatrix} X_1\\ X_2\end{smallmatrix}\right)}_{f}
		\ar[rr]^-{\left(\begin{smallmatrix} \phi_1 \\ \phi_2 \end{smallmatrix}\right)}
		& & {\left(\begin{smallmatrix}Z_1\\ Z_2\end{smallmatrix}\right)}_{h}\ar[rr]^-{\left(\begin{smallmatrix} \psi_1 \\ \psi_2\end{smallmatrix}\right)}& &
		{\left(\begin{smallmatrix}Y_1 \\ Y_2\end{smallmatrix}\right)}_{g}\ar[r]& 0  } \ \    $$
be an almost split sequence in ${\rm H}(\CC)$. Then the sequences $0 \rt X_i\st{\phi_i}\rt Z_i\st{\Psi_i} \rt Y_i\rt 0$, $i=1, 2$, split provided that either of the following situations occur.
		\begin{itemize}
	\item[$(1)$] The terminal term $(Y_1\st{g}\rt Y_2)$ is not of type $(a)$ or $(b)$.
	\item [$(2)$] 	The initial  term $(X_1\st{f}\rt X_2)$ is not of type $(b)$ or $(c)$.
\end{itemize}	
\end{corollary}


We now turn to show that $\Theta$ respects almost split sequences; so we return to the setting that $\CC$ is a dualizing $k$-variety. Let $\mathrm{Y}=(Y_1\st{g}\rt Y_2)$ be an indecomposable object in $\rm{H}(\CC)$ which is not of type $(a)$, $(b)$ or $(c)$. Take
	$$\delta:\,\,\xymatrix@1{  & {\left(\begin{smallmatrix} X_1\\ X_2\end{smallmatrix}\right)}_{f}
		\ar[rr]^-{\left(\begin{smallmatrix} \phi_1 \\ \phi_2 \end{smallmatrix}\right)}
		& & {\left(\begin{smallmatrix}Z_1\\ Z_2\end{smallmatrix}\right)}_{h}\ar[rr]^-{\left(\begin{smallmatrix} \psi_1 \\ \psi_2\end{smallmatrix}\right)}&&
		{\left(\begin{smallmatrix}Y_1 \\ Y_2\end{smallmatrix}\right)}_{g}&  } \ \    $$
to be the almost split sequence in $\rm{H}^{{\rm cw}}(\CC)$ ending at $\mathrm{Y}$.	
For simplicity, set $\mathrm{Z}=(Z_1\st{h}\rt Z_2)$,  $\mathrm{X}=(X_1\st{f}\rt X_2)$, $\phi=(\phi_1, \phi_2)$ and $\psi=(\psi_1, \psi_2)$.	
Note that $\delta$ induces degree-wise split sequences and, applying $\Theta$, one gets the commutative diagram with exact rows
	$$\xymatrix{& 0 \ar[d] & 0 \ar[d] & 0 \ar[d]\\
		0 \ar[r] & K_1 \ar[d] \ar[r] &  K_2 \ar[d]^i
		\ar[r]^{\eta} & K_3 \ar[d]^{\lambda}&\\
		0 \ar[r] &(-, X_1)\ar[d]^{(-,f)} \ar[r] & (-, Z_1)
		\ar[d]^{(-,h)}\ar[r]^{(-,\psi_1)} & (-, Y_1) \ar[d]^{(-,g)} \ar[r] & 0\\
		0 \ar[r] &(-, X_2) \ar[d] \ar[r] & (-, Z_2)
		\ar[d] \ar[r]^{(-, \psi_2)} & (-, Y_2) \ar[d] \ar[r] & 0\\
		& \Theta(\mathrm{X}) \ar[r]^{\Theta(\phi)}\ar[d] & \Theta(\mathrm{Z}) \ar[r]^{\Theta(\psi)}\ar[d]  & \Theta(\mathrm{Y}) \ar[r]\ar[d] &0\\ & 0  & 0  & 0 &  }
	$$
\noindent in $\mmod\CC$ whose bottom row is indeed the image $\Theta(\delta)$ of $\delta$ under the functor $\Theta$.

\begin{lemma}
The map $\eta$ in the above diagram is an epimorphism. As an upshot, $ \Theta(\delta)$ is a short exact sequence in $\mmod \CC.$
\end{lemma}
\begin{proof}
Let $(-, P)\st{ \sigma}\rt K_3\rt 0$ be an epimorphism in $\mmod \CC$ and let $d:P\rt Y_1$ be a morphism in $\CC$ which represents the composite $\la \sigma:(-, P)\rt K_3\rt (-, Y_1).$ Since $gd=0$, Yoneda's lemma induces a morphism $(d, 0):(P\rt 0)\rt Y$ which is plainly not a retraction as ${\rm Y}$ is not of type $(c)$. Hence it must factor over the right almost split map $(\psi_1, \psi_2)$ via, say, $(a, 0)$ for some $a:P\rt Z_1$ in $\CC$. Consequently, the map $(-, a)$ in $\mmod \CC$ satisfies $(-, h)(-, a)=0$. Adding that $(-, P)$ is a projective functor, this gives a map $\gamma: (-, P)\rt K_2$ in $\mmod\CC$ with $(-, a)=i\gamma$. Note that $\lambda\eta\gamma=(-, \psi_1)i\gamma=(-, \psi_1)(-, a)=\lambda\sigma$. But $\lambda$ is a monomorphism; thus $\eta\gamma=\sigma$ whence the surjectivity of $\eta$. The second claim comes up immediately from the Snake Lemma.
\end{proof}

The following theorem is another main result of the section.

\begin{theorem}\label{almostpreserving}
	Under the above notation, $\Theta(\delta)$ is an almost split sequence in $\mmod \CC$.
\end{theorem}
\begin{proof}
The indecomposability of  $\mathrm{X}$ and $\mathrm{Y}$ imply 	that $\Theta(\mathrm{X})$ and $\Theta(\mathrm{Y})$ are indecomposable. By previous lemma, $\Theta(\delta)$ is an exact sequence that, moreover, does not split. Indeed, if it did, then $(-, f)\oplus (-, g)$ would be a minimal projective presentation for $\Theta(\mathrm{Z})$ which should comply with the one provided by the middle column of the above diagram. In view of the form of kernel elements of the functor $\Theta$ declared by Theorem \ref{Theorem-morphism-functor}, an application of Yoneda's lemma gives that, for some objects $A, B$ of $\CC$, there should exist an isomorphism
$$(Z_1\st{h}\rt Z_2)=(X_1\st{f}\rt X_2)\oplus (Y_1\st{g}\rt Y_2)\oplus (A\st{1}\rt A)\oplus(B\rt 0)$$	
of objects in ${\rm H}(\CC)$. As stated before, we may assume  $Z_i\simeq X_i\oplus Y_i$, $i=1, 2$ since $\delta$ belongs to $\CE^{\rm cw}$. However, $\CC$ being a  Krull-Schmidt category implies $A=B=0$ which makes $\delta$ split. This contradiction shows that $\Theta(\delta)$ does not split.

\vspace{.1 cm}

Now, as \cite[Theorem 2.4]{AR74} guarantees that $\mmod\mathcal{C}$ is abelian in this case, invoking \cite[Theorem 2.14]{AR77}, it suffices to show that $\Theta(\psi)$  is right almost split. For let $q:F \rt \Theta(\mathrm{Y})$ be a non-retraction in $\mmod\CC$ and take a projective presentation $(-, W_1)\st{(-, d)}\rt (-, W_2)\rt F\rt 0$ of $F$. Note that by definition, $\Theta(W_1\st{d}\rt W_2)=F.$ The morphism $q$ lifts to a morphism between the projective presentations  $(-, W_1)\st{(-, d)}\rt (-, W_2)\rt F\rt 0$ and $(-, Y_1)\st{(-,g)}\rt (-, Y_2)\rt \Theta(Y)\rt 0$. The lifted morphism induces, again by the Yoneda lemma, a map
		$$\left(\begin{smallmatrix}
	\sigma_1\\ \sigma_2
	\end{smallmatrix}\right):\left(\begin{smallmatrix}
	W_1\\ W_2
	\end{smallmatrix}\right)_d\rt \left(\begin{smallmatrix}Y_1\\  Y_2\end{smallmatrix}\right)_{g}$$
in $\rm{H}(\CC)$ such that $\Theta(\sigma_1, \sigma_2)=q$. Then $(\sigma_1, \sigma_2)$ is not a retraction since otherwise $q$ would be so. Now, $\delta$ being an almost split sequence in ${\rm H}^{\rm cw}(\CC)$,  $(\sigma_1, \sigma_2)$ factors over $\psi$ via some $(\eta_1, \eta_2):(W_1\st{d}\rt W_2)\rt (Z_1\st{h}\rt Z_2)$. Then applying $\Theta$, we see that the morphism $q$ factors over $\Theta(\psi)$ via $\Theta(\eta_1, \eta_2)$.
\end{proof}

\section{The case of wide subcategories}
Our objective in this section is to study the morphism categories raised by functorially finite wide subcategories of $\mmod\Lambda$. Some results from previous section will come to play. Afterwards, we shall switch to functor categories and obtain some results in this direction that extend others from the module category. So let firstly $\CX$ be a functorially finite idempotent-complete subcategory of $\mmod \La$. By \cite[Theorem 2.3]{AS81},  $\CX$ itself is a dualizing variety.

\vspace{.1 cm}

Following Proposition \ref{Prop.exsiyences-AR-comwise}, for  a dualizing $k$-variety $\CC$, there is an equivalence $\tau_{{\rm H}(\CC)}: \underline{{\rm H}^{\rm cw}}(\CC)\rt \overline{{\rm H}^{\rm cw}}(\CC)$ that, based on what we said in previous section, might be considered as an equivalence from ${\rm H}(\CC)/\CV'$ to ${\rm H}(\CC)/\CV$. Pictorially, there exists a composition of equivalences

	$$\xymatrix{
		    & {\rm H}(\CC)/\CV
	 \ar[r]^{\tau^{-1}_{{\rm H}(\CC)}} & {\rm H}(\CC)/\CV'\ar[dr]^{\overline{\Theta'}} && \\
		   \mmod \CC \ar[ur]^{(\overline{\Theta})^{-1}} \ar[rrr] &
		 & &\mmod \CC^{\rm op}\ar[r]^{D}& \mmod \CC  }
	$$
denoted throughout by $\Delta_{\CC}$, or simply by $\Delta$. Applied to the functorially finite subcategory $\CX$ of $\mmod\Lambda$, this yields an equivalence $\Delta_{\CX}:\mmod \CX \rt \mmod \CX$ which is restricted to the category ${\rm proj}\mbox{-}\CX$ of projective functors. Since $\CX$ is idempotent-complete, the Yoneda functor induces an equivalence ${\rm proj}\mbox{-}\CX\simeq \CX$. Summing up, one obtains an equivalence  $\sigma_{\CX}:\CX\rt \CX$ which agrees with the restricted equivalence $\Delta_{\CX}$ via the latter identification.
We notice that, going through the definitions, one figures out that for an object $X$ of $\CX$, there are $A,B\in\CX$ and an exact sequence $$ (B, -) \st{(f, -)}\rt (A, -)\rt D(-, \sigma_{\CX}(X)) \rt 0 $$ in $\mmod \CX$ such that $\tau_{{\rm H}(\CX)}^{-1}(0\rt X)=(A\st{f}\rt B)$.

Recall that a subcategory $\CM$  of $\mmod \La$ is said to be {\em closed under kernels}  (resp. cokernels, images) if for every morphism $X\st{f} \rt Y$ in $\CM$ also ${\rm Ker}(f) $ (resp. $\Coker (f), {\rm Im}(f))$ belongs to $\CM$. Further, $\CM$ is called a {\em wide} subcategory of $\mmod \La$ if it is closed under extensions, kernels and cokernels. It is clear that a wide subcategory is closed under images and is automatically idempotent-complete.

\vspace{.08 cm}

One of the objectives in the sequel is to show that for a functorially finite wide subcategory $\CX$ of $\mmod\Lambda$, $\sigma_\CX$ coincides with the identity functor on $\CX$. A key observation about functorially finite wide subcategories of $\mmod\Lambda$ is that such a subcategory $\CX$ is equivalent, as an exact category, to the category of finitely generated modules over some Artin algebra $\Gamma$. Indeed, according to the constructions given in \cite[Proposition 4.12]{En}, for an Ext-progenerator $P$ of $\CX$, the functor $\Hom_\Lambda(P, -):\mmod\Lambda\rt \mmod \Gamma$ restricts to an exact equivalence $\CX\simeq \mmod \Gamma$ between exact categories, where $\Gamma={\rm End}_\Lambda(P)$. In particular, it follows that functorially finite wide subcategories of $\mmod\Lambda$ admit enough projective and enough injective objects. We thank Haruhisa Enomoto for drawing our attention to \cite{En}.

\vspace{.1 cm}

Denote by $\sigma:=\sigma_{\mmod\Lambda}:\mmod \La\rt \mmod \La$ the auto-equivalence obtained above in the case where $\CX=\mmod\Lambda$. Note that since $\sigma$ is an equivalence, it clearly restricts to an equivalence  $\sigma':{\rm inj}\mbox{-}\La\rt {\rm inj}\mbox{-}\La$ on the subcategory of injective modules. We show that $\sigma'$, and consequently $\sigma$, are nothing but the identity functor on the corresponding categories. We refer e.g. to \cite{HE} for further explanation on how the Auslander-Reiten translation  $\tau_\CH$ and its inverse $\tau_{\CH}^{-1}$ work in this case and suffice to recall that the standard duality functor $D_{\mathcal{H}}$ might be computed in a local manner in terms of the standard duality $D$ of $\Lambda$, i.e., ${\rm D}_{\mathcal{H}}(X\st{f}\rt Y)=({\rm D}(Y)\st{{\rm D}(f)}\rt {\rm D}(X))$.

\begin{lemma}\label{identity-inj}
The restricted equivalence $\sigma'$ is isomorphic to the identity functor on ${\rm inj}\mbox{-}\La.$
\end{lemma}
\begin{proof}
Let $I$ be an injective $\Lambda$-module. There exists a minimal injective resolution in $\CH$
$$0 \rt \left(\begin{smallmatrix} 0\\ I\end{smallmatrix}\right)_0\rt  \left(\begin{smallmatrix} I\\ I\end{smallmatrix}\right)_1\rt \left(\begin{smallmatrix} I\\ 0\end{smallmatrix}\right)_0 \rt 0 $$
of the object $(0\rt I)$. Applying the duality $D_{\mathcal{H}}$ leads to the projective presentation in $\CH^{\rm op}$
$$0 \rt \left(\begin{smallmatrix} 0\\ DI\end{smallmatrix}\right)_0\rt  \left(\begin{smallmatrix} DI\\ DI\end{smallmatrix}\right)_1\rt \left(\begin{smallmatrix} DI\\ 0\end{smallmatrix}\right)_0 \rt 0 $$
\noindent  of the object $D_{\mathcal{H}}(0\rt I)$. Then, we compute the transpose and deduce that $\tau_{\mathcal{H}}^{-1}(0\rt I)\simeq(\nu^{-1}(I)\rt 0)$. As we pointed out earlier in this section, this results in an equivalence $(\nu^{-1}(I), -)\simeq D(-, \sigma(I))$ in  $\mmod (\mmod \La)^{\rm op}$.
Hence, evaluating on the regular module $\La$, yields a natural isomorphism  $\sigma(I)\simeq \nu \nu ^{-1}(I) \simeq I$.
\end{proof}

\begin{proposition}\label{sigma-identity}
For every functorially finite wide subcategory $\CX$ of $\mmod\Lambda$, the equivalence $\sigma_\CX$ is naturally isomorphic to the identity functor on $\CX$.
\end{proposition}
\begin{proof}
According to Lemma \ref{identity-inj},  the restricted equivalence $\sigma'$ is naturally isomorphic to the identity functor on the subcategory of injective $\Lambda$-modules. Using injective resolutions, it is then straightforward to see that the same holds for $\sigma$ itself, that is to say, $\sigma\simeq 1_{\mmod\Lambda}$. Now, for the functorially finite wide subcategory $\CX$ of $\mmod\Lambda$ there exists, according to previous remarks from \cite{En}, an exact equivalence $\CF: \CX\rt \mmod\Gamma$ for a suitable Artin algebra $\Gamma$. It is not difficult to verify that this is compatible with the induced functors $\sigma_{\CX}$ and $\sigma_{\mmod \Gamma}$. Indeed, the functor $\CF$ may naturally be extended to a functor
     $\phi_1:\mmod \CX\to \mmod (\mmod \Gamma)$ by sending any functor $G$ with a presentation 
    $\CX(-, X)\st{(-,f)}\rt \CX(-, Y)\rt G\rt 0$  to the cokernel of the map $$\Hom_\Gamma(-, \CF(f)): \Hom_\Gamma(-, \CF(X))\longrightarrow \Hom_\Gamma(-, \CF(Y)).$$ Likewise, one defines a functor $\phi_4:\mmod (\CX^{\rm op})\rt \mmod (\mmod \Gamma)^{\rm op}$. Moreover, $\CF$ induces a functor 
    $\phi_2:{\rm H}(\CX)/\CV\rt {\rm H}(\Gamma)/\CV $ which is defined by  $(X\st{f}\rt Y) \mapsto (\CF(X)\st{\CF(f)}\rt \CF(Y))$ and also another $\phi_3: {\rm H}(\CX)/\CV'\rt {\rm H}(\Gamma)/\CV'$ with the same rule. Finally, we let $\phi_5:\mmod\CX\rt\mmod(\mmod\Gamma)$ act by sending a functor $G\in\mmod\CX$ with an injective presentation $0\longrightarrow G\longrightarrow D \CX(X, -)\st{D\CX(f, -)}\longrightarrow D \CX(Y, -)$ to the kernel of the map 
$$D\Hom_\Gamma(\CF(f), -): D\Hom_\Gamma(\CF(X), -)\longrightarrow D\Hom_\Gamma(\CF(Y), -) .$$
 Overall, these form a commutative diagram 
    
    
\[\xymatrix{\mmod \CX \ar[r]^{(\bar{\Theta}_{\CX})^{-1}} \ar[d]^{\phi_1} &\frac{{\rm H}(\CX)}{\CV}\ar[r]^{\tau^{-1}_{\rm H(\CX)}}\ar[d]^{\phi_2}&\frac{{\rm H}(\CX)}{\CV'}\ar[r]^{\overline{\Theta'}_{\CX}}\ar[d]^{\phi_3} & \mmod \CX^{\rm op} \ar[r]^D \ar[d]^{\phi_4} & \mmod \CX \ar[d]^{\phi_5} &  \\
	\mmod (\mmod \Gamma)  \ar[r]^(+.6){(\bar{\Theta}_{\Gamma})^{-1}}  &\frac{{\rm H}(\Gamma)}{\CV} \ar[r]^{\tau^{-1}_{\rm H(\Gamma)}}& \frac{{\rm H}(\Gamma)}{\CV'} \ar[r]^(+.3){\overline{\Theta'}_{\Gamma}} & \mmod (\mmod \Gamma)^{\rm op}  \ar[r]^D  & \mmod (\mmod \Gamma)}\]
where the functors $\bar{\Theta}_{\CX}$ and  $\bar{\Theta}_{\Gamma}$ correspond, accoording to what we did in previous section, to the subcategories $\CX$ and $\mmod\Gamma$ respectively of $\mmod\Lambda$ and $\mmod\Gamma$. Since the top and the bottom row respectively define $\Delta_{\CX}$ and $\Delta_{\mmod \Gamma}$, restriction to projective objects gives $F \sigma_\CX\simeq \sigma_{\mmod \Gamma} F$. 
However, $\sigma_{\mmod \Gamma}\simeq 1_{\mmod \Gamma}$ by the above; hence  $\sigma_\CX\simeq 1_\CX$.
\end{proof}

The following couple of propositions are crucial to our next result in this section. We first need to recall the following definition.

\begin{definition} \label{Definition-minimal-proj-inj}
A minimal projective presentation of an object $C\in\CX$ with respect to $\CX$ is an exact sequence $P_1\st{f}\rt P_0\st{h}\rt C$ with $P_1$, $P_0 \in \CP(\CX)$, the class of projective objects of $\CX$, and is computed by taking minimal right $\CP(\CX)$-approximations consecutively. Minimal injective presentations with respect to $\CX$ are defined dually via $\CI(\CX)$, the class of injective objects of $\CX$.
\end{definition}

\begin{proposition}\label{Aslmost-C-0}
Assume $\CX$ is a functorially finite wide subcategory of $\mmod\Lambda$ and $\delta: 0 \rt A \st{f}\rt B \st{g}\rt C\rt 0$ is an almost split sequence in $\CX$.  Let also $ A \st{d}\rt I_0\st{q}\rt I_1$ be a minimal injective presentation  with respect to  $\CX$, where $b: \Coker(d) \rt I_1$ is a minimal left $\CI(\CX)$-approximation, $a: I_0 \rt\Coker(d) $ is the canonical quotient map and $q=ba$. Then the exact sequence

$$\xymatrix@1{ 0 \ar[r] & {\left(\begin{smallmatrix} I_0\\ I_1\end{smallmatrix}\right)}_{q}
		\ar[rr]^-{\left(\begin{smallmatrix} u\\ 1 \end{smallmatrix}\right)}
		& & {\left(\begin{smallmatrix} W\\ I_1\end{smallmatrix}\right)}_{br}\ar[rr]^-{\left(\begin{smallmatrix} v \\ 0\end{smallmatrix}\right)}& &
		{\left(\begin{smallmatrix}C \\ 0\end{smallmatrix}\right)_0}\ar[r]& 0  } \ \    $$

in ${\rm H}(\CX)$ raised by forming the push out diagram
$$\xymatrix{		 A \ar[d]^{d} \ar[r]^f & B \ar[d]^h
	\ar[r]^g& C \ar@{=}[d] && \\
	I_0\ar[d]^a \ar[r]^u & W
	\ar[d]^r \ar[r]^v & C  && \\
	\Coker(d) \ar@{=}[r] & \Coker(d)
	 & &&  }	$$
\noindent in the exact category $\CX$, is almost split.
\end{proposition}
\begin{proof}
Using that $  A \st{d}\rt  I_0\st{s}\rt I_1$ is a minimal injective presentation and $A$ is indecomposable, we deduce that $(I_0\st{s}\rt I_1)$ is indecomposable. Hence it suffices to show that any non-retraction $(\phi, 0):(M\st{p}\rt  N)\rt(C\rt 0)$ in $\CX$ factors over $(v, 0)$. If $\phi$ is a non-retraction, then, since $\delta$ is an almost split sequence, it factors in $\CX$ over $g$ via, say, $w:M \rt B$.  Then it is easy that the morphism $(hw, 0):(M\st{p}\rt N)\rt (W\st{br}\rt I_1)$ factors the morphism $(\phi, 0)$ over $(v, 0)$.
So now take $\phi$ to be a retraction. Without  loss of generality, we may assume $M=C$ and $\phi=1_C$. Two cases might be distinguished:

\vspace{.1 cm}
\noindent Case 1: $p$ is a monomorphism. Since $v$ is a retraction in $\CX$,  there exists $s:C\rt W$ such that $vs=1.$ As $p:C \rt N$ is a monomorphism in $\CX$, there exists an  extension of $brs:C\rt I_1$
to a map $z:N\rt I_1$; that is to say, $zp=brs$. It follows then that $(s, z):(C\st{p}\rt N)\rt (W\st{br}\rt I_1)$ produces the desired factorization.

\vspace{.1 cm}
\noindent Case 2: Assume ${\rm Ker}(p) \neq 0 $. Note that since $\CX$ is a wide subcategory, ${\rm Ker}(p)$ lies in $\CX$. The fact that $(\phi, 0)$ is a non-retraction implies that ${\rm Ker}(p)$ is a proper submodule of $C$ and thus the canonical inclusion $i:{\rm Ker}(p) \rt C$ is a non-retraction in $\CX$. According to the hypothesis, we infer the existence of a map $y:{\rm Ker}(p) \rt B$ such that $gy=i$. Note further that, since $v$ is retraction, one may write $W={\rm Im}(s) \oplus {\rm Ker }(v)$ and, consequently, present $h$ as $h=[l_1, l_2]^t$,  where $l_1:B\rt {\rm Im }(s) $ and $l_2:B\rt {\rm Ker }(v).$ Using the injectivity of ${\rm Ker}(v)$ in $\CX$ yields an extension of $l_2 y:{\rm Ker }(p) \rt  {\rm Ker}(v)$ to $C$; that is, there exists $y':C \rt {\rm Ker}(v) $ such that $y'i=l_2 y$. Putting together, we get a diagram

\[\xymatrix{ & {\rm Ker}(p) \ar[d]^y \ar[r]^i& C \ar[d]^{[s~~y']^t} \ar[r] & {\rm Im }(p) \ \ar[r] &  0 \\  & B\ar[r]^h & W \ar[r]^r & \Coker(h)\ar[r] &  0}\]
\noindent with commutative left part. This induces a map $y'':{\rm Im}(p)\rt {\rm Cok}(h)$ completing the diagram.
Again, as $\CX$ is wide, the monomorphism ${\rm Im}(p) \st{i'}\rt N$ lies inside $\CX$ and, hence, the injectivity of $I_1$ in $\CX$ gives a map $z':N\rt I_1 $ with $z'i'=by''$. Finally, one verifies that the morphism $([s~~y']^t, z'):(C\st{p}\rt N) \rt (W\st{br}\rt I_1)$ gives the required factorization.
\end{proof}

As a dual statement, we record the following proposition.

\begin{proposition}\label{Almost-0-A}
Assume $\CX$ is a functorially finite wide subcategory of $\mmod\Lambda$ and $\delta: 0 \rt A \st{f}\rt B \st{g}\rt C\rt 0$ is an almost split sequence in $\CX$.  Let also $P_1\st{\ell}\rt P_0\st{h}\rt C$ be a minimal projective presentation  with respect to  $\CX$, where $k: P_1\rt \Ker(h)$ is a minimal right $\CP(\CX)$-approximation, $i: \Ker(h)\rt P_0 $ is the canonical inclusion and $\ell=ik$. Then the exact sequence

\[\xymatrix@1{ 0 \ar[r] & {\left(\begin{smallmatrix} 0\\ A\end{smallmatrix}\right)}
		\ar[rr]^-{\left(\begin{smallmatrix} 0\\u \end{smallmatrix}\right)}
		& & {\left(\begin{smallmatrix} P_1\\Z\end{smallmatrix}\right)}_{wk}\ar[rr]^-{\left(\begin{smallmatrix} 1 \\ v \end{smallmatrix}\right)}& &
		{\left(\begin{smallmatrix} P_1 \\ P_0\end{smallmatrix}\right)_{\ell}}\ar[r]& 0  } \]
in ${\rm H}(\CX)$  raised by forming the pull back diagram
$$\xymatrix{		 & \Ker(h)   \ar[d]^w
	\ar@{=}[r]& \Ker(h) \ar[d]^{i} && \\
	A\ar@{=}[d] \ar[r]^{u} & Z
	\ar[d]^r \ar[r]^{v} & P_0\ar[d]^h  && \\
	A \ar[r]^f & B\ar[r]^g
	 &C &&  }	$$
in the exact category $\CX$, is almost split.
\end{proposition}

Any functorially finite wide subcategory $\CX$ of $\mmod\Lambda$ admits almost split sequences by \cite[Theorem 2.4]{AS81}. Hence, following Lemma \ref{ARS duality}, we let $\tau_{\CX}$ denote the Auslander-Reiten translation over $\CX$. The following theorem is another main result in this section.

\begin{theorem}\label{main 1}
Let $\CX$ be a functorially finite wide subcategory of $\mmod \La$  and assume that  $X \in \CX$ is an indecomposable module not belonging to $\CP(\CX)$, and  that $P\st{f}\rt Q\rt X$ is a minimal projective presentation with  respect to $\CX$. Then there is an exact sequence

$$ 0 \rt (-, \tau_{\CX}(X))\rt D(P, -)\rt D(Q, -)\rt D(X, -)\rt 0$$
in $\mmod \CX$
\end{theorem}
\begin{proof}
Proposition \ref{Almost-0-A} yields  that the inverse Auslander-Reiten translation $\tau_{{\rm H}(\CX)}^{-1}(0\rt \tau_{\CX}(X))$ of $(0 \rt \tau_{\CX}(X))$ in ${\rm H}(\CX)$ coincides with $(P\st{f} \rt Q)$. Taking into account our observations on the functor $\sigma_{\CX}$ at the beginning of this section, Proposition \ref{sigma-identity} gives the result.
\end{proof}

In the rest of this section, we will provide some applications of the aforementioned results.

\begin{corollary}\label{corollary exact seq}
Let $M$ be an indecomposable $\Lambda$-module with a minimal projective presentation $P\st{f}\rt Q\rt M\rt 0$. Then there is an exact sequence
$$ 0 \rt (-, \tau(M))\rt D(P, -)\rt D(Q, -)\rt D(M, -) \rt 0$$
in $\mmod (\mmod \La)$.
\end{corollary}

\begin{proof} If $M$ is projective, then such a sequence exists trivially. Otherwise, applying Theorem \ref{main 1} for $\CX=\mmod\Lambda$ settles the statement.
\end{proof}

\begin{remark}
For every non-projective indecomposable $\Lambda$-module $M$, we know that there exists an exact sequence
$$0 \rt \tau(M) \rt \nu(P)\rt \nu(Q)\rt \nu(M)\rt 0$$
where $P\rt Q\rt M\rt 0$ is the minimal projective presentation of $M$. Note that this comes up also by computing the exact sequence of functors in Corollary \ref{corollary exact seq} over $\Lambda$. This shows that in some sense, previous corollary goes parallel to, and generalizes some well-known facts about the Auslander-Reiten theory of $\Lambda$.
\end{remark}

\begin{corollary}
Let $F$ be a functor in $\mmod (\mmod \La)$ with a minimal projective presentation $(-, X)\st{(-, f)}\rt (-, Y)\rt F\rt 0.$ Then there is an exact sequence
$$(Y', -)\st{(g, -)}\rt (X', -)\rt DF\rt 0$$
in $\mmod (\mmod \La)^{\rm op}$ where $(X'\st{g}\rt Y')$ is the inverse Auslander-Reiten translation of $(X\st{f}\rt Y)$ in $\mathcal{H}$.
\end{corollary}
\begin{proof}
Again we specify our constructions to the dualizing variety $\CC=\mmod \La$. By virtue of Proposition \ref{sigma-identity}, the functor $\Delta:=\Delta_{\mmod \La}$ acts identically on projective functors. Using projective presentations, it follows that $\Delta$ is isomorphic to the identity functor on the whole $\mmod (\mmod \La)$. Definition of $\Delta$ then implies that the duality functor $D:\mmod(\mmod\La)\rt \mmod (\mmod \La)^{\rm op}$ is isomorphic to $ \overline{\Theta'}\circ \tau^{-1}_{\CH}\circ (\overline{\Theta})^{-1}$. This proves the claim by following the definitions of the functors involved.
\end{proof}

As a concluding point, let us exploit Corollary \ref{corollary exact seq} to observe a connection between the inverse Auslander-Reiten translation of an indecomposable non-projective $\Lambda$-module $M$ with the second syzygies of injective functors. For, replace $M$ in Corollary \ref{corollary exact seq} by $\tau^{-1}M$ to get the exact sequence
$$ 0 \rt (-, M)\rt D(P, -)\rt D(Q, -)\rt D(\tau^{-1}(M), -) \rt 0$$
in $\mmod (\mmod \La)^{\rm op}$ in which $P\rt Q \rt \tau^{-1}(M) \rt 0$ is a minimal projective presentation. Applying  the duality $D:\mmod (\mmod \La)^{\rm op}\rt \mmod (\mmod \La)$ gives the exact sequence
$$ 0 \rt ( \tau^{-1}(M), -)\rt (Q, -)\rt (P, -)\rt D(-, M) \rt 0.$$
This shows that the functor $( \tau^{-1}(M), -)$ might be interpreted as a second syzygy of the injective functor $D(-, M)$.

\section{simple modules over (stable) Auslander algebra}

Assume $\Lambda$ is of finite representation type and let $M$ be a basic representation generator of $\mmod \La$; that is, $M$ is the direct sum of all pairwise non-isomorphic indecomposable finitely generated $\Lambda$-modules. The endomorphism algebra $A(\Lambda)=\rm{End}_{\La}(M)$, simply denoted by ${\rm A}$ throughout the section, is called the {\it Auslander algebra} of $\La$.  Moreover, the {\it stable Auslander algebra} of $\La$ is by definition $\underline{\rm A}=\rm{End}_{\La}(M)/\CP$, where $\CP$ is the ideal in $\rm{End}_{\La}(M)$ consisting of those endomorphisms factoring through a projective module.
In this case,  we can identify $\mmod {\rm A}$  with $\mmod (\mmod \La) $ via the equivalence induced by the evaluation functor $e_M:\mmod(\mmod \La)\rt \mmod {\rm A}, \ F\mapsto F(M)$. It is also easy to see that $e_M$ induces an equivalence  between $\mmod (\underline{\rm{mod}}\mbox{-}\La)$ and  $\mmod \underline{\rm A}$.

\vspace{.1 cm}

It is known \cite{A76} that indecomposable modules in $\mmod\La$ correspond bijectively to simple functors in $\mmod (\mmod\Lambda)$ by sending  an indecomposable module $M$  to the simple functor $S_M:=(-, M)/\rm{rad}(-, M)$.  Further, for any indecomposable non-projective module $M$, there is a minimal projective resolution
$$0 \rt (-, N)\st{(-, f)}\rt (-, K)\st{(-, g)}\rt (-, M)\rt S_M\rt 0$$
of $S_M$ such that $0 \rt N\st{f}\rt K\st{g}\rt M \rt 0$ is an almost split sequence in $\mmod\Lambda$ (\cite[\S 2]{A76}). Combined to the above observations on the Auslander algebra ${\rm A}$, one may identify simple ${\rm A}$-modules (resp. simple $\underline{\rm A}$-modules) and indecomposable (resp. indecomposable non-projective) $\Lambda$-modules.

\vspace{.1 cm}

Specializing \cite[Construction 3.1]{H} to the  module category $\mmod \La$ gives a functor $\Psi:\CS\rt \mmod (\underline{\rm{mod}}\mbox{-}\La) $, $\CS$ being the monomorphism category of $\Lambda$. This is  defined by sending $(X\st{f}\rt Y)$ in $\CS$ to the functor $F \in \mmod (\underline{\rm{mod}}\mbox{-}\La) $ lying in the exact sequence
$$0 \rt (-, X)\st{(-, f)}\rt (-, Y) \rt (-, \Coker(f)) \rt F\rt 0$$
in  $\mmod (\mmod \La)$. As the following result says, $\Psi$ behaves well with respect to almost split sequences.

\begin{lemma}(\cite[Proposition 5.7]{H})\label{almostspliMono}
Let $0 \rt {\rm U} \rt {\rm V} \rt {\rm W} \rt 0$ be an almost split sequence in $\CS$. Assume ${\rm W}$ is neither of types $(a)$ or $(b)$, nor  of the form $(\Omega(X) \rt P)$, where $X$ is a non-projective indecomposable $\Lambda$-module with projective cover $P$. Then $$0 \rt \Psi({\rm U})\rt \Psi(\rm V) \rt \Psi({\rm W}) \rt 0$$ is an almost split sequence in $\mmod (\underline{\rm{mod}}\mbox{-}\La)$.
\end{lemma}

The following theorem is one of the main results in this section.

\begin{theorem}\label{TMS6}
Assume $\La$ is of finite representation type and $\underline{\rm A}$ is its stable Auslander algebra. Let $S$ be a simple non-projective $\underline{\rm A}$-module. Then, exactly one of the followings hold:

\begin{itemize}
		\item[$(1)$] the Auslander-Reiten translate $\tau_{\underline{\rm A}}(S)$  is projective.
	\item[$(2)$]	there exists a simple $\underline{\rm A}$-module $S'$ such that
	$\tau_{\underline{\rm A}}(S)\simeq \Omega^{-1}_{\underline{\rm A}}(S').$
	In this case, $\Ext^2_{\underline{\rm A}}(S, S')\simeq D\underline{\rm{Hom}}_{\underline{\rm A}}(S, S)$.
	\end{itemize}	
\end{theorem}
\begin{proof}
According to aforementioned remarks, the simple non-projective module $S$ corresponds to a simple functor $(-, C)/\text{rad}(-, C)$ lying in the exact sequence
$$0 \rt (-, A)\st{(-, f)} \rt (-, B)\st{(-, g)}\rt (-, C)\rt (-, C)/\text{rad}(-, C)\rt 0$$
in $\mmod (\mmod \La)$ in such a way that $\lambda: 0 \rt A\st{f}\rt B\st{g  } \rt C\rt 0$ is an almost split sequence in $\mmod \La$. Note that the middle term $B$ may not be projective since otherwise there exists an isomorphism $(-, C)/\text{rad}(-, C)\simeq (-, \underline{C})$ which is against non-projectivity of $S$.

\vspace{.1 cm}

We distinguish two cases: Assume first that $A$ is projective. So by Proposition 3.3 of \cite{HE}, there exists an almost split sequence
	$$\xymatrix@1{0\ar[r] & {\left(\begin{smallmatrix} \text{rad}(A)\\ A\end{smallmatrix}\right)}_{i}
	\ar[r]
	&  {\left(\begin{smallmatrix}A\\ A\end{smallmatrix}\right)}_{1}\oplus{\left(\begin{smallmatrix}\text{rad}(A)\\
		B\end{smallmatrix}\right)}_{fi}\ar[r]&
	{\left(\begin{smallmatrix}A\\B\end{smallmatrix}\right)}_{f}\ar[r]& 0 }$$
in $\CS(\La)$. Hence, in view of Lemma \ref{almostspliMono},  we get the almost split sequence
$$0 \rt \Psi{\left(\begin{smallmatrix} \text{rad}(A)\\ A\end{smallmatrix}\right)}_{i}\rt \Psi {\left(\begin{smallmatrix}\text{rad}(A)\\
		B\end{smallmatrix}\right)}_{fi} \rt \Psi{\left(\begin{smallmatrix}A\\
		B\end{smallmatrix}\right)}_{f}\rt 0$$
in $\mmod (\underline{\rm{mod}}\mbox{-}\La)$. Since $A$ is projective, the definition of $\Psi$ shows that $\Psi ({\rm rad}(A) \st{i} \rt A)\simeq (-, \underline{A/\text{rad}(A)})$. Likewise, as $\lambda$ does not split, we have $\Psi(A\st{f}\rt B)\simeq (-, C)/\text{rad}(-, C)$. Hence $\tau_{\underline{A}}((-, C)/\text{rad}(-, C))\simeq (-, \underline{A/\text{rad}(A)})$ that proves the claim in this case.
		
\vspace{.1 cm}

Assume next that $A$ is not projective. Then there exists an almost split sequence
$$\epsilon: 0 \rt A'\rt B'\rt A \rt 0$$
in  $\mmod \La$. Applying \cite[Lemma 6.3]{H} on $\lambda$ and $\epsilon$,  one infers the almost split sequences
	$$\xymatrix@1{0\ar[r] & {\left(\begin{smallmatrix} A\\ A\end{smallmatrix}\right)}_{1}
	\ar[r]
	&  {\left(\begin{smallmatrix}A\\
		B\end{smallmatrix}\right)}_{f}\ar[r]&
	{\left(\begin{smallmatrix}0\\C\end{smallmatrix}\right)}_{0}\ar[r]& 0 } \ \  \text{and}$$

	$$\xymatrix@1{0\ar[r] & {\left(\begin{smallmatrix}A'\\ I\end{smallmatrix}\right)}_{e}
	\ar[r]
	&  {\left(\begin{smallmatrix}B'\\ I\oplus A\end{smallmatrix}\right)}_{h}\ar[r]&
	{\left(\begin{smallmatrix}A\\A\end{smallmatrix}\right)}_{1}\ar[r]& 0 }$$
in $\CS$ where the second one is obtained from the push-out diagram
$$\xymatrix{		 A' \ar[d]^{e} \ar[r] & B' \ar[d]^h
	\ar[r]& A \ar@{=}[d] & (\dagger)& \\
	I\ar[d]^e \ar[r] & I\oplus A
	\ar[d]^d \ar[r] & A  && \\
	\Omega^{-1}_{\La}(A)  \ar@{=}[r] & \Omega^{-1}_{\La}(A)
	 & &&  }	$$
in which  $e:A'\rt I$ is the injective envelope. From \cite[Lemma 3.3]{HZ}, we can write $(B'\st{h}\rt I\oplus A)\simeq {\rm X}\oplus (J\st{1}\rt J)$, where ${ \rm X}$ is an indecomposable non-projective object and $J$ is either zero or isomorphic to $I$. It follows then that $\tau_{\CS}(A\st{f}\rt B)\simeq {\rm X}$. Accordingly, by taking into account that $(-, C)/\text{rad}(-, C)\simeq\Psi(A\st{f}\rt B)$ by the exact sequence mentioned at the beginning of the proof, another application of  Lemma \ref{almostspliMono}  shows that
\begin{equation}
F:=\tau_{\mmod (\underline{\rm{mod}}\mbox{-}\La)}((-, C)/\text{rad}(-, C))\simeq \Psi({\rm X}).
\end{equation}
However, the definition of $\Psi$ yields $F=\Psi(B'\st{h}\rt I\oplus A)$. Therefore, abusing the notation, we may write $\tau_{\underline{\rm A}}(S)\simeq F$.

\vspace{.1 cm}

Regarding the definition of $\Psi$, the middle column of $(\dagger)$ gives the long exact sequence
	\[
  \footnotesize
  \xymatrix@C=15pt{		
0\ar[r]	&(-, B')\ar[r]&(-, I\oplus A)\ar[r]&(-, \Omega^{-1}_{\La}(A))\ar[rr]\ar[rd]& & \Ext^1_{\La}(-, B')
	\ar[r] & \Ext^1_{\La}(-, I\oplus A)  \\
&&&&F\ar[ru]	&
	&   }
\]
in $\mmod (\mmod \La)$ that implies $F=\Ker(\Ext^1_{\La}(-, B')\rt \Ext^1_{\La}(-, A))$ because $I$ is injective.
On the other hand, since $\epsilon$ is an almost split sequence, our previous considerations show that there exists an exact sequence
\[
  \footnotesize
  \xymatrix@C=10pt{		
	0\ar[r]&(-, A')\ar[r]& (-, B')\ar[r]& (-, A)\ar[rr]\ar[rd] & & \Ext^1_{\La}(-, A')\ar[rr]\ar[rd] & & \Ext^1_{\La}(-, B')
	\ar[r] & \Ext^1_{\La}(-,  A)  \\
	&&&& (-, A)/\text{rad}(-, A)\ar[ur] && F\ar[ur]	&&
	   }
\]
of functors. Invoking \cite[Proposition 7.4]{AR74}, we see that $\Ext^1_{\La}(-, A')$ is an injective functor in $\mmod (\underline{\rm{mod}}\mbox{-}\La)$ and so the induced short exact sequence $0 \rt (-, A)/\text{rad}(-, A)\rt \Ext^1_{\La}(-, A')\rt F \rt 0$ gives $F=\Omega^{-1}_{\underline{\rm A}}((-, A)/\text{rad}(-, A))$. Now it suffices to set $S'=(-, A)/\text{rad}(-, A)$. Notice that the last assertion in the theorem is an upshot of the Auslander-Reiten formula.
\end{proof}

Based on previous theorem, in the following result we establish a bijection between  certain simple modules  over $\underline{A}$ and $\La$. This provides an interesting application concerning the stable equivalences of Artin algebras.

\begin{corollary}\label{CEQ}
Let $\La$ be of finite representation type and $\underline{\rm A}$ be its stable Auslander algebra. There exists a bijection between
\begin{itemize}
\item [$(1)$] the set of isomorphism classes of  non-projective simple modules $S \in \mmod \underline{{\rm A}}$ whose Auslander-Reiten translate $\tau_{\underline{\rm A}} (S)$  is projective; and

\item [$(2)$] the set of isomorphism classes of indecomposable non-injective projective modules $P \in \mmod \La$ such that the middle term of the almost split sequence starting  from $P$ is not projective; and

\item [$(3)$] the set of isomorphism classes of simple modules $S \in \mmod \La$ whose projective cover $P(S)$ is non-injective, and the middle term of the almost split sequence starting from $P(S)$ is not projective.
\end{itemize}
\end{corollary}
\begin{proof}
The bijection between $(2)$ and $(3)$ might be shown by restricting the well-known bijection between simple and indecomposable projective modules. The map from $(2)$ to $(1)$ is given by sending $P$ to $(-, \tau^{-1}(P))/{\rm rad}(-, \tau^{-1}(P))$, which is well-defined due to the argument given in Theorem \ref{TMS6}. Let now $S$ be a simple non-projective module in $\mmod \underline{\rm A}$ with $\tau_{\underline{\rm A}}(S)$ projective. We already know that there is an indecomposable non-projective $\Lambda$-module $C$ such that $S\simeq (-, C)/{\rm rad}(-, C)$. We claim that $\tau(C)$ is projective; otherwise, as in Theorem \ref{TMS6}, there exists a simple $\underline{\rm A}$-module $S'$ such that $\tau_{\underline{\rm A}}(S)\simeq \Omega^{-1}_{\underline{\rm A}}(S')$. Hence the short exact sequence $0 \rt S' \rt I\rt \Omega^{-1}_{\underline{\rm A}}(S')\rt 0$, in which $I'$ is the injective envelop of $S'$, splits. This means that $\Omega_{\underline{\rm A}}^{-1}(S')=0$ and so $\tau_{\underline{\rm A}}(S)=0$ which is against non-projectivity of $S$. Thus $\tau(C)$ is projective and setting $P:=\tau(C)$ completes the proof.
\end{proof}

Recall that two Artin algebras $\La$ and $\La'$ are said to be stably equivalent if there is an equivalence of categories $\underline{{\rm mod}}\mbox{-}\Lambda\simeq\underline{{\rm mod}}\mbox{-}\Lambda'$.
Denote by $n(\La)$ the number of iso classes of simple $\Lambda$-modules satisfying the third condition of the above corollary. As a byproduct, we show that $n(\Lambda)$ is an invariant of the stable equivalences.

\begin{proposition}
Let $\La$ and $\La'$ be of finite representation type and stably equivalent. Then $n(\La)=n(\La')$.
\end{proposition}
\begin{proof}
Since $\La$ and $\La'$ are stably equivalent, it follows that the corresponding stable Auslander algebras $\underline{\rm A}$ and $\underline{\rm A'}$ are Morita equivalent. By Corollary \ref{CEQ}, we see that simple modules in $\mmod \La$  (resp. $\mmod \La'$) that satisfy condition $(3)$ correspond bijectively to non-projective simple modules over the stable Auslander algebra $\underline{\rm A}$ (resp. $\underline{\rm A'}$) with projective Auslander-Reiten translates. We are done since the modules of latter type are preserved under Morita equivalences.
\end{proof}

\color{black}


The following lemma is taken from \cite[Theorem 3.2]{HE}.

\begin{lemma}\label{almosSplittwo}
	Let $\delta: \  0 \rt A \st{f}\rt B \st{g}\rt C \rt 0$ and $\delta': \ 0 \rt A'\st{f'}\rt B'\st{g'}\rt A\rt 0$ be  almost split sequences in $\mmod\Lambda$. Then
	$$\xymatrix@1{0\ar[r] & {\left(\begin{smallmatrix} B'\\ A\end{smallmatrix}\right)}_{g'}
		\ar[rr]^-{\left(\begin{smallmatrix}\left[\begin{smallmatrix} g'\\1
			\end{smallmatrix}\right]\\\left[\begin{smallmatrix} 1\\ f
			\end{smallmatrix}\right]
			\end{smallmatrix}\right)}
		& & {\left(\begin{smallmatrix}A\\ A\end{smallmatrix}\right)}_{1}\oplus{\left(\begin{smallmatrix}B'\\
			B\end{smallmatrix}\right)}_{fg'}\ar[rr]^-{\left(\begin{smallmatrix}\left[\begin{smallmatrix} -1 &g'
			\end{smallmatrix}\right]\\\left[\begin{smallmatrix} -f &1
			\end{smallmatrix}\right]
			\end{smallmatrix}\right)}& &
		{\left(\begin{smallmatrix}A\\B\end{smallmatrix}\right)}_{f}\ar[r]& 0 },$$
	is an almost split sequence  in $\mathcal{H}$. Further, $\left(\begin{smallmatrix}B'\\
	B\end{smallmatrix}\right)_{fg'}$ is  an indecomposable object.
\end{lemma}

The following theorem should be served as the second main result of this section.

\begin{theorem}\label{Theorem 5.2}
Assume $\La$ is a self-injective algebra  of finite representation type and let ${\rm A}$ be its  Auslander algebra. Let also $S$ be a simple  ${\rm A}$-module of projective dimension two. Then there exists a simple ${\rm A}$-module $S'$ of projective dimension two such that
$\Omega_{{\rm A}}(S')\simeq \tau^{-1}_{{\rm A}}(S)$. In this case, $\Ext^2_{{\rm A}}(S', S)\simeq D\overline{\rm{Hom}}_{{\rm A}}(S, S)$.
\end{theorem}

\begin{proof}
We may identify the simple module $S$ by the simple functor $(-, A)/\text{rad}(-, A)$  lying in the exact sequence
$$0 \rt (-, A')\st{(-, f')} \rt (-, B')\st{(-, g')}\rt (-, A)\rt (-, A)/\text{rad}(-, A)\rt 0 \ \ \ (\dagger)$$
in $\mmod (\mmod \La)$ in such a way that $\delta: 0 \rt A'\st{f'}\rt B'\st{g'} \rt A\rt 0$ is an almost split sequence in $\mmod \La$. Let also $\delta': \ 0 \rt A\st{f}\rt B\st{g}\rt C\rt 0$ be  an almost split sequence in $\mmod \La$. Then by Lemma \ref{almosSplittwo} there exists an almost split sequence
	$$\xymatrix@1{0\ar[r] & {\left(\begin{smallmatrix} B'\\ A\end{smallmatrix}\right)}_{g'}
		\ar[rr]^-{\left(\begin{smallmatrix}\left[\begin{smallmatrix} g'\\1
			\end{smallmatrix}\right]\\\left[\begin{smallmatrix} 1\\ f
			\end{smallmatrix}\right]
			\end{smallmatrix}\right)}
		& & {\left(\begin{smallmatrix}A\\ A\end{smallmatrix}\right)}_{1}\oplus{\left(\begin{smallmatrix}B'\\
			B\end{smallmatrix}\right)}_{fg'}\ar[rr]^-{\left(\begin{smallmatrix}\left[\begin{smallmatrix} -1 &g'
			\end{smallmatrix}\right]\\\left[\begin{smallmatrix} -f &1
			\end{smallmatrix}\right]
			\end{smallmatrix}\right)}& &
		{\left(\begin{smallmatrix}A\\B\end{smallmatrix}\right)}_{f}\ar[r]& 0 }$$
in $\mathcal{H}$. Thanks to Theorem \ref{almostpreserving}, this induces the almost split sequence

$$0 \rt \Theta(B'\st{g'}\rt A)\rt \Theta(B'\st{fg'}\rt B)\rt \Theta(A\st{f}\rt B)\rt 0 \ \ \  (\dagger\dagger)$$
in $\mmod(\mmod\Lambda)$. Note that $(\dagger)$ implies $\Theta(B'\st{g'}\rt A)=(-, A)/{\rm rad}(-, A)=S$. Hence by $(\dagger\dagger)$,  $\tau^{-1}_{{\rm A}}(S)=\Theta(A\st{f}\rt B)$. Set now ${\rm W}=(A\st{f}\rt B)$. Then, by definitions, there exists an exact sequence
 \[
  \footnotesize
  \xymatrix@C=15pt{		
	(-, A)\ar[r]& (-,B)\ar[rr]\ar[rd]&&(-, C)\ar[rr]\ar[rd]& & \Ext^1_{\La}(-, A)
	\ar[r] & \Ext^1_\Lambda(-, B).  \\
	&&\Theta({\rm W})\ar[ur]&& (-, C)/{\rm rad}(-, C)\ar[ru]	&
	&   }
\]
Set $S'$ be the simple functor $(-, C)/{\rm rad}(-, C)$. Then the short exact sequence
$0 \rt \Theta({\rm W})\rt (-, C)\rt S'\rt 0$ proves the claim.
\end{proof}

\section{Auslander-Reiten components of Auslander algebras}
Unless otherwise specified, we assume throughout this section that  $\La$ is a non-semisimple indecomposable self-injective algebra of finite representation type, and ${\rm A}$ denotes its Auslander algebra. In the whole section, we use the identification $\mmod{\rm A}\simeq \mmod(\mmod\Lambda)$ described earlier. Once more, in this section, the quadruple family of objects in $\CH$ of types $(a), (b), (c),$ and $(d)$ become important. We aim to identify certain components of the (stable) Auslander-Reiten quiver of ${\rm A}$. To this end, we need firstly study particular $\tau_\CH$-periodic objects in $\CH$ and their periodicity.

\subsection{$\tau_{\CH}$-periodic objects} As we observed in  Theorem \ref{almostpreserving}, the functor $\Theta:\CH\rt \mmod {\rm A}$ behaves well with respect to  almost split sequences in the sense that if there exists an almost split sequence $0 \rt {\rm X}\rt {\rm Y}\rt {\rm Z} \rt 0$ in $\CH$ where ${\rm Z}$ is not of types  $(a)$, $(b)$ or $(c)$, then
$0 \rt \Theta({\rm X})\rt \Theta({\rm Y})\rt \Theta({\rm Z}) \rt 0$
is also an almost split sequence in $\mmod {\rm A}$. Also we have seen in Theorem  \ref{Theorem-morphism-functor} that one is given an equivalence
$\CH/\CV \simeq \mmod{\rm A}$, where $\CV$ is generated by the objects of type $(b)$ or $(c)$. Therefore, the Auslander-Reiten quiver $\Gamma_{{\rm A}}$ of the Auslander algebra ${\rm A}$ might be computed via the Auslander-Reiten quiver $\Gamma_{\CH}$ of $\CH$ by removing vertices corresponding to iso-classes of indecomposable objects of either types $(b)$ or $(c)$.

\vspace{.1 cm}

The following construction is vital for the rest of this section. It is mainly based on an analysis of various almost split sequences already obtained in \cite{HE}. For the sake of brevity, we prefer not to rewrite most of them here and suffice to give the precise reference number therein.

\begin{construction} \label{CM7} Let $C$ be an indecomposable non-projective $\Lambda$-module. There exist almost split sequences $\epsilon_1: 0 \rt \tau(C) \st{f}\rt B\st{g} \rt C \rt 0$ and $\epsilon_2: 0 \rt \tau^2(C)\st{f'} \rt B'\st{g'} \rt \tau(C)\rt 0$ in $\mmod \La$. Applying  Lemmas \ref{almosSplittwo} and \ref{LemmaBoundary} we deduce the almost split sequences


$$\xymatrix@1{0\ar[r] & {\left(\begin{smallmatrix} B'\\ \tau(C)\end{smallmatrix}\right)}_{g'}
		\ar[rr]^-{\left(\begin{smallmatrix}\left[\begin{smallmatrix} g'\\1
			\end{smallmatrix}\right]\\\left[\begin{smallmatrix} 1\\ f
			\end{smallmatrix}\right]
			\end{smallmatrix}\right)}
		& & {\left(\begin{smallmatrix}\tau(C)\\ \tau(C)\end{smallmatrix}\right)}_{1}\oplus{\left(\begin{smallmatrix}B'\\
			B\end{smallmatrix}\right)}_{fg'}\ar[rr]^-{\left(\begin{smallmatrix}\left[\begin{smallmatrix} -1 &g'
			\end{smallmatrix}\right]\\\left[\begin{smallmatrix} -f &1
			\end{smallmatrix}\right]
			\end{smallmatrix}\right)}& &
		{\left(\begin{smallmatrix}\tau(C)\\B\end{smallmatrix}\right)}_{f}\ar[r]& 0 },$$
$$\xymatrix@1{  0\ar[r] & {\left(\begin{smallmatrix} \tau(C)\\ \tau(C)\end{smallmatrix}\right)}_{1}
			\ar[rr]^-{\left(\begin{smallmatrix} 1 \\ f\end{smallmatrix}\right)}
			& & {\left(\begin{smallmatrix}\tau(C)\\ B\end{smallmatrix}\right)}_{f}\ar[rr]^-{\left(\begin{smallmatrix} 0 \\ g\end{smallmatrix}\right)}& &
			{\left(\begin{smallmatrix}0\\ C\end{smallmatrix}\right)}_{0}\ar[r]& 0, } $$	and
$$\xymatrix@1{  0\ar[r] & {\left(\begin{smallmatrix} \tau(C)\\ 0\end{smallmatrix}\right)}_{0}
			\ar[rr]^-{\left(\begin{smallmatrix} f \\ 0 \end{smallmatrix}\right)}
			& & {\left(\begin{smallmatrix} B\\ C\end{smallmatrix}\right)}_{g}\ar[rr]^-{\left(\begin{smallmatrix} g \\ 1\end{smallmatrix}\right)}& &
			{\left(\begin{smallmatrix}C \\ C\end{smallmatrix}\right)}_{1}\ar[r]& 0 } $$
in $\CH$.  Evidently, the indecomposable object $(B'\st{fg'}\rt B)$ is not projective; so let ${\rm X}:=\tau_{\CH}(B'\st{fg'}\rt B)$. Also, as $(B'\st{g'}\rt \tau(C))$ is not projective, we let ${\rm Y}:=\tau_{\CH}(B'\st{g'}\rt \tau(C))$ and note that ${\rm X}$ and ${\rm Y}$ are not projective. In view of \cite[Propositions 2.2, 4.1]{HE}, there exists an almost split sequence

 $$\xymatrix@1{  0\ar[r] & {\left(\begin{smallmatrix} \nu(P)\\ \nu(Q)\end{smallmatrix}\right)}_{\nu(h)}
			\ar[rr]
			& & {\rm Y} \oplus {\left(\begin{smallmatrix} I\\ 0\end{smallmatrix}\right)}_0 \ar[rr]& &
			{\left(\begin{smallmatrix}\tau^2(C)\\ 0\end{smallmatrix}\right)}_0\ar[r]& 0 } $$
in $\CH $ where $P\st{h}\rt Q \rt \tau^2(C)\rt 0$ is the minimal projective presentation, and $I$ is an injective module.
On the other hand, by Proposition 2.4 and Theorem 4.2 of \cite{HE}, we have the almost split sequence

 $$\xymatrix@1{  0\ar[r] & {\left(\begin{smallmatrix} 0 \\  \tau \nu \tau ^2(C)\end{smallmatrix}\right)_0}
			\ar[rr]
			& & \tau_{\CH}({\rm Y}) \oplus {\left(\begin{smallmatrix} 0\\ P\end{smallmatrix}\right)_0} \ar[rr]& &
			{\left(\begin{smallmatrix}\nu(P)\\ \nu(Q)\end{smallmatrix}\right)}_{\nu(h)}\ar[r]& 0 }  $$
in $\CH$ where $P$ is projective. Putting all together, one obtains the mesh
 	
		$$\xymatrix@-5mm{&&	\tau({\rm X})\ar[dr] &&{\rm X}\ar[dr]\ar@{.>}[ll] && {{\left(\begin{smallmatrix} B'\\ B\end{smallmatrix}\right)}_{fg'}}\ar[dr]\ar@{.>}[ll]& \\	
		&\tau({\rm Y})\ar[dr]\ar[ur]&	&{\rm Y}\ar[ur]\ar[dr]\ar@{.>}[ll]&& {{\left(\begin{smallmatrix} B'\\ \tau (C)\end{smallmatrix}\right)}_{g'}}\ar[dr]\ar[ur]\ar@{.>}[ll]&& {{\left(\begin{smallmatrix} \tau(C)\\ B\end{smallmatrix}\right)}_{f}} \ar[dr]\ar@{.>}[ll]\\
	{{\left(\begin{smallmatrix} 0 \\ \nu \tau^3(C)\end{smallmatrix}\right)_0}}\ar[ur]	&&	{{\left(\begin{smallmatrix} \nu(P)\\ \nu (Q)\end{smallmatrix}\right)}_{\nu(h)}}\ar[ru]\ar@{.>}[ll]&&	{{\left(\begin{smallmatrix} \tau^2(C)\\ 0\end{smallmatrix}\right)_0}}\ar[ur]\ar@{.>}[ll]&& {{\left(\begin{smallmatrix} \tau(C)\\ \tau(C)\end{smallmatrix}\right)}_{1}} \ar[ur]\ar@{.>}[ll]&&\ar@{.>}[ll]{{\left(\begin{smallmatrix} 0\\ C\end{smallmatrix}\right)_0}}	}$$	
in the Auslander-Reiten quiver $\Gamma_{{\rm A}}$ of ${\rm A}$ in which the vertices $(I\rt 0)$ and $(0\rt P)$ have been ignored.
\end{construction}

Let us recall from \cite[Remark 5.7]{HE} that $\mathscr{A}=\nu\tau^3$ defines an auto-equivalence on the stable category $\underline{\mmod}\Lambda$. As is expected, the $\mathscr{A}$-orbit of an indecomposable non-projective $\Lambda$-module $M$ consists of the modules $\mathscr{A}^m(M)$ where $m$ ranges over the integer numbers.

\begin{proposition}\label{4mperiodicityMorphism}
Let $\Lambda$ be self-injective (not necessarily of finite representation type) and suppose every indecomposable non-projective $\Lambda$-module possesses a finite $\mathscr{A}$-orbit.  Then any indecomposable non-projective object ${\rm X}$ in $\CH$ of either types  $(a)$, $(b)$, $(c)$ or $(d)$ is of $\tau_{\CH}$-periodicity a multiple of $4$.
\end{proposition}
\begin{proof}
According to our previous observations, all mentioned objects lie in the $\tau_{\CH}$-orbit of some indecomposable object of type $(a)$. Hence it suffices to prove the statement only for  ${\rm X}=(0\rt N)$ with $N$ an indecomposable non-projective module. Justified by the hypothesis, choose a least integer $n$ with $\mathscr{A}^n(N)=N$.   Considering the particular mesh in $\Gamma_{{\rm A}}$ as illustrated in Construction \ref{CM7}, we get $\tau_{\CH}^4 (0 \rt N )\simeq (0 \rt \mathscr{A}(N))$ and thus $\tau_{\CH}^{4n}(0\rt N)\simeq(0\rt \mathscr{A}^n(N))=(0\rt N)$.
\end{proof}
	
Based on previous proposition, we are now able to prove the following theorem which will prove useful later on.

\begin{theorem}\label{PS2P}  Assume $\Lambda$ is self-injective (not necessarily of finite representation type) and every indecomposable non-projective $\Lambda$-module possesses a finite $\mathscr{A}$-orbit. Then every simple ${\rm A}$-module  of projective dimension $2$ is $\tau_{\rm A}$-periodic of periodicity divided by $4$.
\end{theorem}
\begin{proof}
Recall that such simple ${\rm A}$-modules might be identified with simple functors $S_M=(-, M)/{\rad}(-, M)$ where $M$ is an indecomposable non-projective $\Lambda$-module lying in an almost split sequence $0 \rt \tau(M)\st{g}\rt N\st{f}\rt M\rt $ in $\mmod \La$. By Proposition \ref{4mperiodicityMorphism}, $(0\rt \tau^{-1}(M))$ is of $\tau_{\CH}$-periodicity $4n$ for a suitable integer $n$. Since $(0\rt \tau^{-1}(M))$ and $(M\st{1}\rt M)$ lie in the same $\tau_{\CH}$-orbit, it follows that  $( M\st{1}\rt M)$ is also of the same periodicity $4n$. Thus the irreducible morphism $(N\st{f}\rt M)\rt (M\st{1}\rt M)$ in $\Gamma_{\CH}$ remains fixed after $4n$ applications of $\tau_{\CH}$ and accordingly, $(N\st{f}\rt M)$ should be of $\tau_{\CH}$-periodicity $4n$. Consequently, according to Theorem \ref{almostpreserving}, $\tau^{4n}_{A}(S_M)=\tau^{4n}_{A}\Theta(N\st{f}\rt M)=\Theta\tau^{4n}_{\CH}(N\st{f}\rt M)=\Theta(N\st{f}\rt M)=S_M$.
\end{proof}

\vspace{.1 cm}




\subsection{Modules $M$ with $\tau(M)=\Omega(M)$}

\begin{definition}
Let $M$ be an indecomposable non-projective module.  We say $M$ has the property $(*)$ if $0 \rt \Omega(M)\rt P(M)\rt M\rt 0$ is an almost split sequence in $\mmod \La$ where $P(M)$ is the projective cover of $M$.
\end{definition}

Modules satisfying this property have already been classified in \cite[Theorem V.3.3]{AuslanreitenSmalo}:  these are exactly non-injective simple $\Lambda$-modules  $M$ that are not a composition factor of ${\rm rad}(I)/{\rm soc}(I)$ for every injective $\Lambda$-module $I$.
This clearly yields that such modules are necessarily $\mathscr{A}$-periodic.
Note also that in the situation of the definition, $\tau(M)=\Omega(M)$. The goal in this subsection is to see that existence of modules with this property may heavily affect the shape of the AR-quiver of ${\rm A}$ and in particular cases may even make ${\rm A}$ into an algebra of finite representation type. As a first pace to study modules with property $(*)$, the following lemma shows that this property carries over from a module to its (co)syzygies.

\begin{lemma}\label{Lemma 7.7}
Let $M$ be an indecomposable non-projective $\Lambda$-module. If $M$ has the property $(*)$, then so do all its syzygies (resp. cosyzygies).  In particular, the short exact sequences

$$ 0 \rt \Omega^{i+1}(M)\rt P^i\st{}\rt \Omega^i(M)\rt 0 \ \ \text{for} \ i\geq 0, \ \text{and}$$
$$0 \rt \Omega^{i}(M)\rt I^i \rt \Omega^{i-1}(M)\rt 0 \ \ \text{for} \ i\leq 0 $$
in $\mmod \La$ induced by the minimal projective (resp. injective) resolution of $M$ are almost split.
\end{lemma}
\begin{proof}
We prove the lemma for integers $i\geq 0$ by using an inductive argument whose basis $i=0$ is satisfied by the assumption; so we put $i> 0$. Consider  the almost split sequence $0 \rt \tau\Omega^i(M)\rt B\rt \Omega^i(M)\rt 0$ in $ \mmod \La$. We claim that $B$ is projective. Assume to the contrary that $B$ has a non-projective indecomposable direct summand $C$. The induction hypothesis then implies that  $\tau^{-1}(C)$ is a non-projective direct summand of $P^{i-1}$, which is absurd. Now use the fact that the morphisms involved in an almost split sequence are minimal to deduce that	 $\tau\Omega^i(M)=\Omega^{i+1}(M)$.
\end{proof}


The following lemma shows a property of the modules $M$ for which $(*)$ is satisfied; this will be used later on in this section.

\begin{lemma}\label{Lemma-nu}
Under the hypothesis of Lemma \ref{Lemma 7.7}, one has $\nu(P^{i+1})\simeq P^i$ for $i\geq 0$ and $\nu^{-1} (I^{i-1}) \simeq I^i$ for $i\leq 0$.
\end{lemma}
\begin{proof}
We prove the first assertion. The minimal projective presentation $P^1\rt P^0\rt M \rt 0$ induces the short exact sequence
$0 \rt \tau(M)\rt \nu(P^1)\rt \nu(P^0) \rt \nu(M)\rt 0$ in $\mmod\Lambda$. Note that, as $\nu$ is an auto-equivalence of $\mmod\La$, the map $\tau(M)\rt \nu(P^1)$ is minimal and thus defines the injective envelope of $\tau(M)$ as $\nu(P^1)$ is injective.
However, by definition, the monomorphism $\tau(M)\rt P^0$ obtained by composing the isomorphism $\tau(M)\simeq\Omega(M)$ and the inclusion $\Omega(M)\rt P^0$ is also minimal with $P^0$ injective. Therefore $\nu(P^1)\simeq P^0$ since the injective envelope is unique up to isomorphism. Now we deduce the result by applying an inductive argument in conjunction with Lemma \ref{Lemma 7.7}.
\end{proof}



We also need the following lemma which is probably well-known; nonetheless we sketch a proof for the sake of completeness. 

\begin{lemma}\label{Lemma-proj-inj}
    Let $F$ be an object in $\mmod(\mmod \La)$. Then $F$ is projective-injective if and only if $F\simeq (-, I)$, for some injective $\Lambda$-module $I$.
\end{lemma}
\begin{proof}
Set $F\simeq(-, I)$ for an injective $\Lambda$-module $I$. We only need to show that $\Ext^1(G, (-, I))=0$ for an arbitrary $G\in\mmod(\mmod\Lambda)$. Since the global dimension of $\mmod(\mmod\Lambda)$ is at most $2$, one may take a projective resolution 
$0\rt(-, C)\st{(-,f)}\longrightarrow (-,B)\st{(-, g)}\longrightarrow (-,A)\longrightarrow G\rt 0$ of $G$ which is, according to Yoneda's Lemma, induced by an exact sequence $0\rt C\st{f}\rt B\st{g}\rt A$ of $\Lambda$-modules. Now take a $\Lambda$-map $h:B\rt I$ such that $hf=0$. This induces a map $\im(g)\rt I$ which extends to a $\Lambda$-map $\bar{h}:A\rt I$ with $h=\bar{h}g$ due to injectivity of $I$. This shows that applying $(-, I)$ on the aforementioned projective resolution leaves it exact at $(-, B)$, whence the result.  

\vspace{.05 cm}
If, conversely, $F$ is projetive-injective, then $F\simeq (-, I)$ for some $\Lambda$-module $I$. To see that $I$ is injective, one takes a short exact sequence $0\rt C\rt B\rt A\rt 0$ of $\Lambda$-modules and uses arguments like those applied above, along with Yoneda's Lemma to deduce that $F$ leaves the latter sequence exact. 
\end{proof}

The following theorem is the promised one.

\begin{theorem} \label{Pstar}
Assume there exists an indecomposable non-projective $\Lambda$-module $M$ with the property $(*)$. Then the full subquiver of the Auslander-Reiten quiver $\Gamma_{{\rm A}}$ of ${\rm A}$ which is obtained by removing projective-injective vertices, is a finite oriented cycle. In particular, the Auslander algebra ${\rm A}$ is of finite representation type.
\end{theorem}
\begin{proof}
The minimal projective presentation
\[ \cdots \rt P^{n}\st{w_{n}}\rt P^{n-1}\rt \cdots P^1\st{w_1}\rt P^0\rt M\rt 0\]
of $M$ induces, according to Lemma \ref{Lemma 7.7}, the almost split sequences
\[\epsilon_i:  0 \rt \Omega^{i+1}(M)\st{v_i}\rt P^i \st{u_i}\rt \Omega^i(M)\rt 0\]
in $\mmod\Lambda$. Applying  Lemma \ref{almosSplittwo} on $\epsilon_0$ and $\epsilon_1$ gives the almost split sequence

\[\xymatrix@1{  0\ar[r] & {\left(\begin{smallmatrix} P^1\\ \Omega(M)\end{smallmatrix}\right)}_{u_1}
			\ar[rr]
			& & {\left(\begin{smallmatrix} \Omega(M)\\ \Omega(M)\end{smallmatrix}\right)}_{1} \oplus {\left(\begin{smallmatrix} P^1\\ P^0\end{smallmatrix}\right)}_{w_1} \ar[rr]& &
			{\left(\begin{smallmatrix}\Omega(M)\\ P^0\end{smallmatrix}\right)}_{v_0}\ar[r]& 0. } \]
Note that, by Lemma \ref{LemmaBoundary}, $\tau_{\mathcal{H}}({\left(\begin{smallmatrix} 0\\ M\end{smallmatrix}\right)})={\left(\begin{smallmatrix} \Omega(M)\\ \Omega(M)\end{smallmatrix}\right)}_{1}$ and $\tau_{\mathcal{H}}\Big({\left(\begin{smallmatrix} \Omega(M)\\ \Omega(M)\end{smallmatrix}\right)}_{1}\Big)={\left(\begin{smallmatrix} \Omega^2(M)\\ 0\end{smallmatrix}\right)}$.
Moreover, in light of \cite[Proposition 2.4]{HE}, we get $\tau_{\CH}(P^1\st{w_1}\rt P^0)=(0 \rt \Omega(M))$ and so there exists an almost split sequence

 \[\xymatrix@1{  0\ar[r] & {\left(\begin{smallmatrix} 0\\ \Omega(M)\end{smallmatrix}\right)}_{0}
			\ar[rr]
			& & {\left(\begin{smallmatrix} P^1\\ \Omega(M)\end{smallmatrix}\right)}_{u_1} \oplus {\left(\begin{smallmatrix} 0\\ P^0\end{smallmatrix}\right)}_{0} \ar[rr]& &
			{\left(\begin{smallmatrix}P^1\\ P^0\end{smallmatrix}\right)}_{w_1}\ar[r]& 0. }\]
Furthermore, an application of \cite[Theorem 3.5]{HE} provides us with another almost split sequence

\[\xymatrix@1{  0\ar[r] & {\left(\begin{smallmatrix} \Omega^2(M)\\ P^1 \end{smallmatrix}\right)}_{v_1}
			\ar[rr]
			& & {\left(\begin{smallmatrix} 0\\ \Omega(M)\end{smallmatrix}\right)}_{0} \oplus {\left(\begin{smallmatrix} P^1\\ P^1\end{smallmatrix}\right)}_{1}\oplus {\left(\begin{smallmatrix} \Omega^2(M)\\ 0\end{smallmatrix}\right)}_{0} \ar[rr]& &
			{\left(\begin{smallmatrix}P^1\\ \Omega(M)\end{smallmatrix}\right)}_{u_1}\ar[r]& 0. }\]
Also  \cite[Proposition 2.2]{HE} combined to Lemma \ref{Lemma-nu} results in the almost split sequence

\[ \xymatrix@1{  0\ar[r] & {\left(\begin{smallmatrix} \nu(P^3)\\ \nu(P^2)\end{smallmatrix}\right)}_{\nu(w_3)}
			\ar[rr]
			& & {\left(\begin{smallmatrix} \Omega^2(M)\\ P^1\end{smallmatrix}\right)}_{v_1} \oplus {\left(\begin{smallmatrix} \nu(P^3)  \\ 0 \end{smallmatrix}\right)}_{0} \ar[rr]& &
			{\left(\begin{smallmatrix}\Omega^2(M)\\ 0\end{smallmatrix}\right)}_{0}\ar[r]& 0. }\]
It is easy to see that, as $\Omega^2(M)$ satisfies $(*)$, so does $\nu\Omega^2(M)$ and consequently, $\tau \nu\Omega^2(M)=\Omega\nu\Omega^2(M)\simeq \nu\Omega^3(M)=\mathscr{A}(M)$. Therefore, by \cite[Proposition 2.4]{HE}, we have $\tau_{\CH}(\nu(P^3) \st{\nu(w_3)}\rt \nu(P^2))=(0 \rt \mathscr{A}(M))$. Continuing in this manner, one obtains the following mesh in the Auslander-Reiten quiver of $\mathcal{H}$.

\label{DigaramSpices}
 		{\tiny 
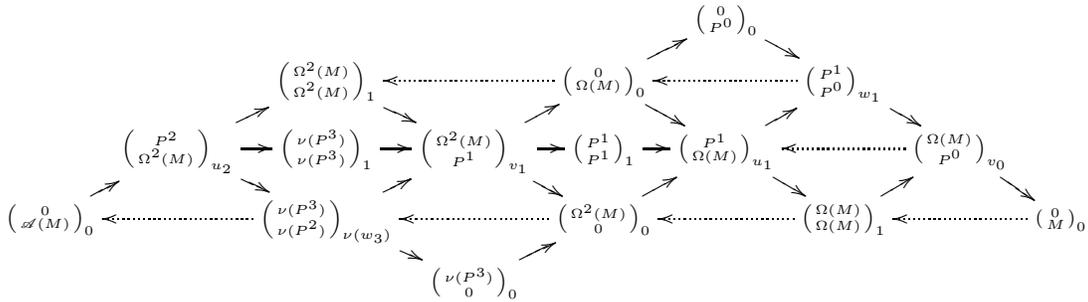
\[\xymatrix@-7mm{&&&&& {\left(\begin{smallmatrix}0\\ P^0 \end{smallmatrix}\right)}_{0}\ar[dr]&&&	&\\
 	&&{\left(\begin{smallmatrix}\Omega^2(M)\\ \Omega^2(M)\end{smallmatrix}\right)}_{1}\ar[dr]	&&{\left(\begin{smallmatrix}0\\ \Omega(M)\end{smallmatrix}\right)}_{0}\ar[ur]\ar[dr]\ar@{.>}[ll] && {{\left(\begin{smallmatrix} P^1\\ P^0\end{smallmatrix}\right)}_{w_1}}\ar[dr]\ar@{.>}[ll]&& \\
&{\left(\begin{smallmatrix}P^2 \\ \Omega^2(M)\end{smallmatrix}\right)}_{u_2}\ar[ur]\ar[dr]\ar[r]& {\left(\begin{smallmatrix} \nu(P^3)\\ \nu (P^3)\end{smallmatrix}\right)}_{1}\ar[r]&{\left(\begin{smallmatrix}\Omega^2(M) \\  P^1\end{smallmatrix}\right)}_{v_1}\ar[r]\ar[ur]\ar[dr]&{\left(\begin{smallmatrix} P^1\\ P^1\end{smallmatrix}\right)}_{1}\ar[r]& {{\left(\begin{smallmatrix} P^1\\ \Omega(M) \end{smallmatrix}\right)}_{u_1}}\ar[dr]\ar[ur]&& {{\left(\begin{smallmatrix} \Omega(M)\\ P^0\end{smallmatrix}\right)}_{v_0}} \ar[dr]\ar@{.>}[ll]&\\
	{{\left(\begin{smallmatrix} 0\\ \mathscr{A}(M) \end{smallmatrix}\right)}_{0}}\ar[ur]&&{{\left(\begin{smallmatrix} \nu(P^3)\\ \nu (P^2)\end{smallmatrix}\right)}_{\nu(w_3)}}\ar[ru]\ar[dr]\ar@{.>}[ll]&&	{{\left(\begin{smallmatrix} \Omega^2(M)\\ 0\end{smallmatrix}\right)}_{0}}\ar[ur]\ar@{.>}[ll]&& {{\left(\begin{smallmatrix} \Omega(M)\\ \Omega(M)\end{smallmatrix}\right)}_{1}} \ar[ur]\ar@{.>}[ll]&&\ar@{.>}[ll]{{\left(\begin{smallmatrix} 0\\ M\end{smallmatrix}\right)}_{0}}\\&&&{\left(\begin{smallmatrix} \nu(P^3)  \\ 0 \end{smallmatrix}\right)}_{0}\ar[ru]&&&&&	}	
	\]}

Starting then with the vertex $(0 \rt \mathscr{A}(M))$ and iterating the above arguments, one may calculate the vertices lying on the left side of $(0\rt M)$ in $\Gamma_{\CH}$. Also, by considering the minimal injective resolution of $M$, the vertices on the right part appear. Summarizing, it follows that the component in $\Gamma_{\CH}$ containing  the vertex $(0 \rt M)$ is obtained by putting together all parts of the above shape corresponding to modules in the $\mathscr{A}$-orbit of $M$. By virtue of  previous considerations, $\Gamma_{{\rm A}}$ comes up from $\Gamma_{\CH}$ by removing  vertices of types $(b)$ and $(c)$. On the other hand,  $(0\rt P^0)$ corresponds under the equivalence  $\frac{\CH}{\CV}\simeq \mmod(\mmod\Lambda)$  to the functor $(-, P^0)$ which is , by Lemma \ref{Lemma-proj-inj}, the only projective-injective vertex in the above mesh. Hence, after removing  projective-injecive vertices, we observe that the remaining full subquiver of the AR-quiver $\Gamma_{{\rm A}}$ is obtained by gluing together all pieces of the following shape.
	
 	{\tiny 
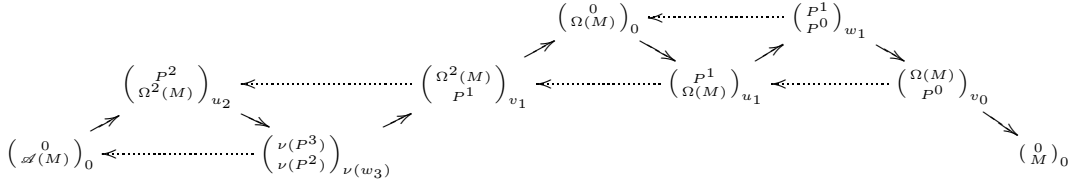
\[\xymatrix@-7mm{&&&&& &&&	&\\
 	&&	&&{\left(\begin{smallmatrix}0\\ \Omega(M)\end{smallmatrix}\right)}_{0}\ar[dr] && {{\left(\begin{smallmatrix} P^1\\ P^0\end{smallmatrix}\right)}_{w_1}}\ar[dr]\ar@{.>}[ll]&& \\
&{\left(\begin{smallmatrix}P^2 \\ \Omega^2(M)\end{smallmatrix}\right)}_{u_2}\ar[dr]& &{\left(\begin{smallmatrix}\Omega^2(M) \\  P^1\end{smallmatrix}\right)}_{v_1}\ar[ur]\ar@{.>}[ll]&& {{\left(\begin{smallmatrix} P^1\\ \Omega(M) \end{smallmatrix}\right)}_{u_1}}\ar[ur]\ar@{.>}[ll]&& {{\left(\begin{smallmatrix} \Omega(M)\\ P^0\end{smallmatrix}\right)}_{v_0}} \ar[dr]\ar@{.>}[ll]&\\
	{{\left(\begin{smallmatrix} 0\\ \mathscr{A}(M) \end{smallmatrix}\right)}_{0}}\ar[ur]&&{{\left(\begin{smallmatrix} \nu(P^3)\\ \nu (P^2)\end{smallmatrix}\right)}_{\nu(w_3)}}\ar[ru]\ar@{.>}[ll]&&	&&  &&{{\left(\begin{smallmatrix} 0\\ M\end{smallmatrix}\right)}_{0}}	}	
	\]}	
Now, since $M$ is $\mathscr{A}$-periodic, we just get a finite oriented cycle, as required. That ${\rm A}$ is of finite representation type follows from \cite[Theorem VII.2.1]{AuslanreitenSmalo}.
\end{proof}

\subsection{Components of the stable Auslander-Reiten quiver of {\rm A}}
We let $\Gamma^s_{\rm A}$, the stable Auslander-Reiten quiver of ${\rm A}$, be the subquiver of $\Gamma_{\rm A}$ obtained by removing projective vertices and their $\tau_{\rm A}$-orbits. It should be clarified that here, we distinguish with a usual custom in the corresponding literature where this terminology applies while removing vertices that are both projective and injective. Also we notice that, generally, this has nothing to do with the Auslander-Reiten quiver of the stable Auslander algebra $\underline{\rm A}$. For instance, despite $\underline{{\rm A}}$ which is self-injective in this case, ${\rm A}$ can not be self-injective since it is of global dimension $2$.

\vspace{.05 cm}

Below, we use results from \cite{AS93} to get a nice intuition of the stable Auslander-Reiten quiver $\Gamma_{\rm A}^s$ of ${\rm A}$ in terms of the AR quiver of a triangulated category.

\begin{remark}\label{RMark-Frob}
Recall from Section  $3$ that the class $\CX$ in $\CH$ consisting of all objects of type $(a), (b), (c)$, or $(d)$ determines an exact structure ${\bf E3}=\CH_{\CX} $ on $\CH$ which has enough projectives and enough injectives. We claim that $\CP(\CH_{\CX})=\CX\cup {\rm proj}\mbox{-}\CH$ and  $\CI(\CH_{\CX})=\tau_{\CH}(\CX)\cup {\rm inj}\mbox{-}\CH$, the subcategories of projectives and injectives of $\CH_{\CX}$, coincide. Indeed, as in Construction \ref{CM7}, $\tau_{\mathcal{H}}(\mathcal{X})\subseteq\mathcal{X}$. Since every indecomposable injective object in $\CH$ is of type $(b)$ or $(c)$, we get $\CI(\CH_{\CX})\subseteq \CP(\CH_{\CX})$. To settle the reverse inclusion, note that  by \cite[Proposition 3.6]{HE}, for an indecomposable projective $\Lambda$-module $P$ there exists a projective $\Lambda$-module $Q$ with $\tau_{\CH}(Q \rt 0)=(0\rt P)$. Besides that projective objects of type $(P\st{1}\rt P)$ lie in ${\rm inj}\mbox{-}\CH$, this ensures that ${\rm proj}\mbox{-}\CH\subseteq \CI(\CH_{\CX})$. 

So, to get $\CP(\CH_{\CX}) \subseteq \CI_{\CX}(\CH_{\CX})$, we need only show $\CX \subseteq \CI(\CH_{\CX}) $. As $\CI(\CH_{\CX})$ contains all injectives, it suffices to show that every non-injective object in $\CX$ lies inside $\CI(\CH_{\CX})$.
Take an indecomposable non-injective object ${\rm M}$ of $\CX$. If ${\rm M}$ is of type $(a)$, then ${\rm M}=(0\rt M)$ for an indecomposable $\Lambda$-module $M$. If $M$ is projective then as above,  ${\rm M}=\tau_\CH(Q\rt 0)$ for some $Q$; otherwise $M$ is non-injective and Proposition $2.4$ of \cite{HE} shows that ${\rm M}=\tau_\CH(P_1\rt P_0)$ where $P_1\rt P_0\rt\tau^{-1}(M)\rt 0$  is the minimal projective presentation. If ${\rm M}=(M\st{1}\rt M)$ is of type $(b)$ with $M$ non-injective, then Lemma \ref{LemmaBoundary} gives ${\rm M}\simeq\tau_\CH(0\rt\tau^{-1}(M))$. Furthermore, if ${\rm M}=(M\rt 0)$ is of type $(c)$, then again Lemma \ref{LemmaBoundary} shows that ${\rm M}\simeq\tau_\CH(\tau^{-1}(M)\st{1}\rt \tau^{-1}(M))$ as $M$ is non-injective. Finally, if ${\rm M}=(P\st{f}\rt Q)$ is of type $(d)$ then, setting $N=\Coker(\nu(f))$,  we deduce from \cite[Proposition 2.2]{HE} that ${\rm M}=\tau_\CH(N\rt 0)$. Summarizing, these imply that $\CX\subseteq\CI(\CH_{\CX})$ and the above claim follows; that is to say, $\CH_{\CX}$ is a Frobenius exact category and, consequently, the stable category $\underline{\CH}_{\CX}$ is  triangulated.

\vspace{.05 cm}

On the other hand, according to \cite[Proposition 1.9]{AS93},  we infer that an almost split sequence $0 \rt {\rm X}\rt {\rm Y}\rt {\rm Z}\rt 0$ in $\CH$ is an almost split sequence in $\CH_{\CX}$ if and only if neither ${\rm X} \in \CI({\CH}_{\CX})$ nor ${\rm Z} \in \CP({\CH}_{\CX})$. Thus, in order to get the Auslander-Reiten quiver of the triangulated category $\underline{\CH}_{\CX}$, it is enough to remove the iso-classes of indecomposable objects in $\CX$  and arrows attached to them from the Auslander-Reiten quiver $\Gamma_{\CH}$ of $\CH$. But, as stated before,  what remains after deleting, is exactly the stable Auslander-Reiten quiver $\Gamma^s_{A}$ of ${\rm A}$.
\end{remark}

The following theorem is the main result in this subsection. Note that if $\Lambda$ admits a module $M$ with property $(*)$ then, according to Theorem \ref{Pstar}, $\Gamma^s_{{\rm A}}$ is just a set of single vertices. That's why one has to exclude this case from the hypothesis below.

\begin{theorem}\label{prop Xi}
Assume $\La$ is indecomposable self-injective  of finite representation type and $\Xi$ is a component of $\Gamma^s_{\rm A}$ containing a simple module $S_M$ for an indecomposable non-projective $\Lambda$-module $M$ not fulfilling $(*)$. Then
	\begin{itemize}
\item [$(i)$]	If $\Xi$ is finite, then $\Xi=\mathbb{Z}\Delta/G$, where $\Delta$ is a Dynkin quiver and $G$ is an automorphism group of $\mathbb{Z}\Delta$ containing a positive power of the translation. Moreover, $\Xi$ is $\Gamma^s_{\rm A}$ itself if we further assume that $\La$ is indecomposable.
	\item [$(ii)$] If $\Xi$ is infinite, then it is a stable tube.
\end{itemize}
\end{theorem}
\begin{proof}
We have seen before that the AR-quiver $\Gamma_{{\rm A}}$ of the Auslander algebra ${\rm A}$ is obtained from $\Gamma_{\CH}$ by removing the vertices of types $(b)$ and $(c)$. Note that indecomposable objects of type $(a)$ correspond to indecomposable projective ${\rm A}$-modules and those of type $(d)$ lie in the $\tau_{{\rm A}}$-orbit of indecomposable projective ${\rm A}$-modules by Construction \ref{CM7}. So in fact,  $\Xi$ emerges by deleting vertices of either types $(a)$, $(b)$, $(c)$ and $(d)$ and the arrows attached to them from a component $\Xi'$ of $\Gamma_{\CH}$ containing the vertex $(0\rt M)$. Note that such a $\Xi$ is connected as $M$ does not satisfy the property $(*)$.  Note also that all vertices in $\Gamma^s_{{\rm A}}$ are stable in the sense that $\tau^m_{\rm A}(-)$ is well-defined over them for arbitrary integers $m$. Therefore Theorem \ref{PS2P} implies that the vertex $S_M$ in $\Xi$ is $\tau_{{\rm A}}$-periodic and both the assertions in $(i)$ and $(ii)$ follow from \cite[Theorem 5.5]{L}. For the second statement in $(i)$, note that indecomposability of $\Lambda$ implies that the lower triangular matrix algebra
$T_2(\Lambda)=\left(\begin{smallmatrix} \Lambda & 0\\
\Lambda & \Lambda\end{smallmatrix}\right)$
is also indecomposable and recall that $\mathcal{H}$ is naturally equivalent to the category $\mmod T_2(\Lambda)$. Now if $\Xi$ is  finite, then so is $\Xi'$. As such,  $\Xi'$ itself is a finite component of $\Gamma_{\mathcal{H}}$ which, according to \cite[Theorem VII.2.1]{AuslanreitenSmalo}, should be the whole of {$\Gamma_{\CH}$}. Hence  $\Xi=\Gamma^s_{\rm A}$, as desired.
\end{proof}

\vspace{.2 cm}

We conclude this section by quoting an observation from \cite{HE}. Assume  $M$ is an indecomposable non-projective $\Lambda$-module. Denote by $[M]_{\mathscr{A}}$  the $\mathscr{A}$-orbit  of $M$.  Let also $\Gamma_{\CH}(M)$  be the unique component of $\Gamma_{\CH}$ containing the vertex $(0\rt M)$. Moreover, we set

$$\mathcal{T}=\{\Gamma_{\CH}(M)\mid M\ \text{is indecomposable non-projective}\}$$

$$\mathcal{E}=\{[M]_{\mathscr{A}}\mid M\ \text{is indecomposable non-projective}\} $$

\noindent Then there exists a well-defined map $\delta:\mathcal{E}\rt \mathcal{T}$  which is given by sending $[M]_{{\mathscr{A}}}$ to $\Gamma_{\CH}(M)$.
Further, if we let $\mathcal{T}_{\infty}$  denote the subset of $\mathcal{T}$  consisting of all infinite components, and $\mathcal{E}_{\infty}$ be the inverse image of  $\mathcal{T}_{\infty}$  under $\delta$, then by \cite[Proposition 5.8]{HE}, $\delta$ is surjective and the restricted map  $\delta\mid:\mathcal{E}_{\infty}\rt \mathcal{T}_{\infty}$  is a  bijection whenever $\Lambda$ is indecomposable self-injective of finite representation type.  When $\CH$ is of infinite representation type, one has $\CT_{\infty}=\CT$ and $\CE_{\infty}=\CE$.

\vspace{.1 cm}
Inspired by this result, we let $\mathcal{L}$ be the set of all  components of $\Gamma_{{\rm A}}$  containing a simple module. Define $\la:\mathcal{T}\lrt\mathcal{L}$ by  $\lambda(\Gamma_{\CH}(M))=\Gamma_{{\rm A}}(S_M)$ where $\Gamma_{{\rm A}}(S_M)$  is the component of $\Gamma_{{\rm A}}$ that contains the simple vertex $S_M$.

\begin{proposition} Suppose $\Lambda$ is indecomposable self-injective  of finite representation type.
\begin{itemize}
\item[$(i)$] The map $\la$ is well-defined and surjective.
\item[$(ii)$] $\lambda$ restricts to a bijection $\la\mid:\mathcal{T}_{\infty}\rt \mathcal{L}_{\infty}$, where $\mathcal{L}_{\infty}$ is the subset of $\mathcal{L}$ consisting of all infinite components.
\item[$(iii)$] The sets $\mathcal{E}_{\infty}$ and $\mathcal{L}_{\infty}$ are in bijection.
    \end{itemize}
\end{proposition}
\begin{proof}
$(i)$. Take $M_0$ and $M_1$ to be indecomposable non-projective $\Lambda$-modules with $\Gamma_{\CH}(M_0)=\Gamma_{\CH}(M_1)$. Then there exist almost split sequences $0 \rt A_0\st{f_0}\rt B_0\st{g_0}\rt M_0\rt 0$ and $0 \rt A_1\st{f_1}\rt B_1\st{g'_1}\rt M_1\rt 0$ in $\mmod \La$. By Lemma \ref{LemmaBoundary}, there exist almost split sequences
 \[\label{equationAl1}
 \xymatrix@1{  0\ar[r] & {\left(\begin{smallmatrix} A_i\\ 0\end{smallmatrix}\right)}_{0}
			\ar[rr]^-{\left(\begin{smallmatrix} f_i \\ 1\end{smallmatrix}\right)}
			& & {\left(\begin{smallmatrix}B_i\\ M_i\end{smallmatrix}\right)}_{g_i}\ar[rr]^-{\left(\begin{smallmatrix} 0 \\ g_i\end{smallmatrix}\right)}& &
			{\left(\begin{smallmatrix} 0\\ M_i\end{smallmatrix}\right)}_{0}\ar[r]& 0 }\]
in $\mathcal{H}$ for $i=0,1$.
Accordingly, $(B_0\st{g_0}\rt M_0)$ and $(B_1\st{g_1}\rt M_1)$ lie inside $\Gamma_{\CH}(M_0)$ and the connectedness of $\Gamma_{\CH}(M_0)$ gives the existence of a walk $(B_0\st{g_0}\rt M_0)\longleftrightarrow x_1 \longleftrightarrow\cdots x_{n-1}\longleftrightarrow (B_1\st{g_1}\rt M_1)$ in $\Gamma_{\mathcal{H}}$. If none of the $x_i$ is of the form  $(b)$ or $(c)$, then using Theorem \ref{almostpreserving}, an application of the functor $\Theta$ gives a walk in $\Gamma_{{\rm A}}$ between $\Theta(B_0\st{g_0}\rt M_0)=S_{M_0}$ and $\Theta(B_1\st{g_1}\rt M_1)=S_{M_1}$. Consequently, $\lambda(\Gamma_{\CH}(M_0))=\lambda(\Gamma_{\CH}(M_1))$. Otherwise we may, without loss of generality, assume that the $x_i$ are all non-projective and apply the arguments used in Construction \ref{CM7} to obtain a walk in $\Gamma_{\CH}$ passing through $(B_i\st{g_i}\rt M_i)$, $i=0, 1$, none of the vertices over which are of the forms $(b)$ or $(c)$. To settle $(ii)$, assume $\Gamma_{\rm A}(S_M)$ is an infinite component of $\Gamma_{\rm A}$  for an indecomposable non-projective $\Lambda$-module $M$. By Theorem \ref{prop Xi}, $\Gamma_{\rm A}(S_M)$ is a stable tube and the $\tau_{\rm A}$-orbit of $S_M$ generates the mouth of $\Gamma_{\rm A}(S_M)$. The fact that the mouth of a stable tube is unique reveals that $\Gamma_{\rm A}(S_M)$ is uniquely determined by $\Gamma_{\CH}(M)$. The last statement is a combination of $(ii)$ and \cite[Proposition 5.8]{HE}.
\end{proof}

\end{document}